\documentclass[10pt]{article}

\usepackage{amsmath}
\usepackage{amsfonts}
\usepackage{mathtools}
\usepackage{amssymb}
\usepackage{here}
\usepackage{fullpage}
\usepackage{mathrsfs}
\usepackage{enumerate}
\usepackage{multirow}
\usepackage{dsfont}
\usepackage{color}
\usepackage{setspace}
\usepackage[numbers,compress]{natbib}

\providecommand{\blue}[1]{\color{black}{#1}\color{black}\hspace{0pt}}

\everymath{\displaystyle}

\newtheorem{theorem}{Theorem}

\newtheorem{proposition}[theorem]{Proposition}

\newtheorem{definition}[theorem]{Definition}

\newtheorem{remark}[theorem]{Remark}

\providecommand{\blue}[1]{\textcolor[rgb]{0.00,0.00,1.00}{#1}}

\DeclareMathOperator*{\col}{col}
\DeclareMathOperator*{\row}{row}

\DeclareMathOperator*{\diag}{diag}

\DeclareMathOperator*{\esssup}{ess\ sup}

\providecommand{\BDeltab}[1]{\ensuremath{\mathcal{B}^{#1}_{\hspace{-1pt}\boldsymbol{\Delta}}}}

\newenvironment{proof}{{\it Proof :~}}{\hfill$\diamondsuit$\\}

\title{Stability and performance analysis of linear positive systems with delays using input-output methods}

\author{Corentin Briat\thanks{Swiss Federal Institute of Technology--Z\"{u}rich, Department of Biosystems Science and Engineering, Basel, Switzerland. Email: briatc@bsse.ethz.ch,corentin@briat.info; url: http://www.briat.info}}

\date{}

\begin{document}

\maketitle


\begin{abstract}
It is known that input-output approaches based on scaled small-gain theorems with constant $D$-scalings and integral linear constraints are non-conservative for the analysis of some classes of linear positive systems interconnected with uncertain linear operators. This dramatically contrasts with the case of general linear systems with delays where input-output approaches provide, in general, sufficient conditions only. Using these results we provide simple alternative proofs for many of the existing results on the stability of linear positive systems with discrete/distributed/neutral time-invariant/-varying delays and linear difference equations.  In particular, we give a simple proof for the characterization of diagonal Riccati stability for systems with discrete-delays and generalize this equation to other types of delay systems. The fact that all those results can be reproved in a very simple way demonstrates the importance and the efficiency of the input-output framework for the analysis of linear positive systems. The approach is also used to derive performance results evaluated in terms of the $L_1$-, $L_2$- and $L_\infty$-gains. It is also flexible enough to be used for design purposes.
\end{abstract}


\section{Introduction}



Positive systems \citep{Farina:00} are a class of systems that are able to represent important processes arising, among others, in epidemiology, biology, biochemistry, ecology; see e.g. \citep{Farina:00,Murray:02,Briat:09h,Briat:12c,Briat:13i,Briat:15e,Briat:16a}. They also naturally arise in the design of interval observers, a class of observers whose error dynamics is purposely governed by a positive system and which allows to estimate upper- and lower-bounds on the state of the system; see e.g. \cite{Gouze:00,Mazenc:11,Briat:15g,Efimov:15,Efimov:16,Efimov:16b}. Finally, they can be used as comparison systems for the analysis of more complex systems, notably, for the analysis of systems with delays; see e.g. \cite{Ngoc:16,Mazenc:16}. Besides these applicability properties, they have been shown to exhibit very interesting theoretical properties. For instance, structured state-feedback controllers and certain instances of the static output-feedback controllers can be designed in a non-conservative way by solving tractable linear programs \citep{Aitrami:07,AitRami:11,Briat:11h}. The $L_1$-, $L_2$- and $L_\infty$-gains of such systems can be also easily characterized in terms of linear \citep{Briat:11g,Briat:11h,Ebihara:11,Rantzer:16} or semidefinite programs \citep{Tanaka:10}. The robust stability analysis of such systems subject to parametric uncertainties can be exactly performed using scaled small-gain results with constant $D$-scalings \cite{Briat:11h,Colombino:15} or integral linear constraints \citep{Briat:11g,Briat:11h,Briat:15cdc}, the latter being the linear counterpart of the integral quadratic constraints \citep{Megretski:93,RantzerMegretski:97}. Finally, it also got recently proved that the scaled-small gain theorem in the $L_2$-framework states a necessary and sufficient condition for the stability of interconnections in the special case of positive systems \citep{Colombino:15}, a fact that does not hold true for general linear systems affected by time-invariant parametric uncertainties; see e.g. \citep{Packard:93}. A possible workaround to this problem is to consider instead the $L_\infty$-framework \citep{Dahleh:95,Khammash:93} where the scaled-small gain theorem with constant $D$-scalings states a necessary and sufficient condition for the robust stability of linear systems.

The influence of delays on the dynamics of linear positive systems and certain classes of nonlinear monotone systems have been well studied and several necessary and sufficient conditions for the stability have been obtained using various approaches; see e.g. \citep{Haddad:04,AitRami:09,Briat:11h,Mason:12,Zhu:15,Shen:15}. We propose here to reprove many of the existing result pertaining on linear systems using a different approach, namely, using \emph{input-output approaches} and, more specifically, using scaled-small gain results with $D$-scalings specialized to linear positive systems \cite{Colombino:15,Briat:11g,Briat:11h} and integral linear constraints results \cite{Briat:11g,Briat:11h,Briat:15cdc}. \blue{Albeit popular (see e.g. \cite{Zhang:01a,Niculescu:01,Knospe:03,GuKC:03,GouaiPeau:06,KnospeR:06,Gouaisbaut:06,Kao:07,Ariba:09,Ariba:10,Gouaisbaut:11,Briat:book1,Fridman:14,Zhu:15,Zhu:15b,Li:16b}), input/output methods do not seem to have been applied so far for the analysis of linear positive systems with delays.}
%
%
%
We notably show that the following statements are rather immediate consequences of scaled-small gain results and integral  linear constraint results:
\begin{enumerate}[(i)]
  \item A linear positive system with discrete constant time-delay is stable if and only if the same system with the delay set to 0 is also stable \citep{Haddad:04,Briat:11h}.
  \item A linear positive system with bounded discrete time-varying delay is stable if and only if the same system with constant delay is also stable \citep{AitRami:09,Briat:11h}. This is generalized in  \citep{Shen:14} to the case of time-varying distributed delays and to the case of arbitrarily large discrete-delays in \citep{Shen:15}.
  \item A linear positive system with constant discrete delay is stable if and only if the associated Riccati equation has diagonal solutions \citep{Mason:12,Aleksandrov:16}.
 \blue{ \item A linear positive coupled differential-difference equation with a single time-varying discrete delay is stable if and only if the same system with the delay set to 0 is also stable \citep{Shen:15b}.}
 \item A linear positive system with discrete time-varying delays is stable if and only if two conditions (which will stated later), known to be only necessary for the stability of general time-delay systems, are satisfied \citep{Zhu:15}.
 \item A linear positive system with distributed time-varying delay is stable if and only if the sum of the matrix acting on the non-delayed state and the integral of the distributed-delay kernel is Hurwitz stable \cite{Shen:14}.
 \blue{\item A linear positive neutral system is stable if and only if the system with zero delay is also stable and it is strongly stable \cite{Ebihara:16,Ebihara:17}. In particular, it is shown that the strong stability of the difference equation together with the stability of the retarded part is equivalent to the stability of the neutral delay equation.}
\end{enumerate}
In this regard, the contribution of the paper is not only the development of some new stability results but also to provide a different, simple and flexible approach for the analysis of linear positive systems with delays. The approach can then be extended to cope with additional uncertainties (e.g. additional parametric uncertainties, sector-nonlinearities, etc.) and can be used for design purposes (e.g. for the design of interval observers \cite{Gouze:00,Mazenc:11,Briat:15g,Efimov:15,Efimov:16,Efimov:16b}).\\

\blue{\noindent\textbf{Outline.} The structure of the paper is as follows: definitions and preliminary results are given in Section \ref{sec:preliminary}. General stability results for uncertain linear positive systems are presented in Section \ref{sec:general} and are applied to linear positive systems with discrete delays in Section \ref{sec:discrete_main}, to linear positive delay-difference equations in Section \ref{sec:difference}, to linear positive coupled differential-difference equations in Section \ref{sec:coupled}, to linear positive systems with distributed delays in Section \ref{sec:distributed_main}, to linear positive neutral systems in Section \ref{sec:neutral}.}\\

\noindent\textbf{Notations.} The cone of positive and nonnegative vectors of dimension $n$ are denoted by $\mathbb{R}_{>0}^n$ and $\mathbb{R}_{\ge0}^n$, respectively. The set of positive integers is given by $\mathbb{Z}_{>0}$. For two real full matrices $A,B$ having the same dimension, the inequalities $A>(\ge)B$ are componentwise while for two real symmetric matrices $A,B$ having the same dimension, the relation $A\prec B$ means that $A-B$ is negative definite. We denote the set of $n\times n$ positive definite diagonal matrices by $\mathbb{D}^n_{\succ0}$. We denote by $\rho(A)$ the spectral radius of the square matrix $A$. The $n$-dimensional vector of ones is denoted by $\mathds{1}_n$. For a vector $v\in\mathbb{R}^n$, $||v||_p$ denotes the standard vector $p$-norm while for a matrix $M\in\mathbb{R}^{n\times m}$, $\textstyle||M||_p:=\max_{||v||_p=1}||Mv||_p$ is the matrix induced $p$-norm. \blue{For some matrices  $M_1,
\ldots,M_n$ of appropriate dimensions, we define $\textstyle\row_{i=1}^N\{M_i\}:=\begin{bmatrix}
  M_1 & \ldots & M_N
\end{bmatrix}$ and $\textstyle\col_{i=1}^N\{M_i\}:=\begin{bmatrix}
  M_1^T & \ldots & M_N^T
\end{bmatrix}^T$.}

\section{Preliminaries}\label{sec:preliminary}

\subsection{System definition}

Let us consider the following linear system:
\begin{equation}\label{eq:mainsyst_gen}
\begin{array}{rcl}
  \dot{x}(t)&=&Ax(t)+Ew(t),\  x(0)=x_0\\
    z(t)&=&Cx(t)+Fw(t)\\
\end{array}
\end{equation}
where $x,x_0\in\mathbb{R}^n$, $w\in\mathbb{R}^{q}$ and $z\in\mathbb{R}^{q}$ are the state of the system, the initial condition, the input and the output, respectively. When $x_0=0$, the above system defines a linear time-invariant convolution operator $\Sigma:w\mapsto z$ given by
\begin{equation}\label{eq:convolution}
  z(t)=\int_0^th(s)w(t-s)ds
\end{equation}
where $h(t)=Ce^{At}E+F\delta(t)$ where $\delta(t)$ is the Dirac distribution and whose transfer function is given by
\begin{equation}
  \widehat{\Sigma} (s):=C(sI-A)^{-1}E+F.
\end{equation}
We then have the following proposition \citep{Farina:00}:
\begin{proposition}
The following statements are equivalent:
\begin{enumerate}[(i)]
  \item The system \eqref{eq:mainsyst_gen} is (internally) positive; i.e. for any $x_0\ge0$ and any $w(t)\ge0$, we have that $x(t)\ge0$ and $z(t)\ge0$ for all $t\ge0$.
  \item The matrix $A$ is Metzler (i.e. all the off-diagonal elements are nonnegative)  and the matrices $E,C,F$ are nonnegative (i.e. all the entries are nonnegative).
\end{enumerate}
\end{proposition}

\subsection{Norms and gains}

Let us start with the definition of the $L_p$-norms for signals \citep{Desoer:75a}:
\begin{definition}
Let $w:\mathbb{R}_{\ge0}\to\mathbb{R}^n$, then its $L_p$-norm is given by
\begin{equation}
  ||w||_{L_p}:=\left\{\begin{array}{lcl}
    \displaystyle\left(\int_0^\infty||w(t)||_p^p dt\right)^{1/p}&&\textnormal{when }p\in\mathbb{Z}_{>0}\\
    \esssup_{t\ge0}||w(t)||_\infty&&\textnormal{when }p=\infty\\
  \end{array}\right..
\end{equation}
We say that $w\in L_p$ if $||w||_{L_p}$ is finite. 
\end{definition}

The $L_p$-gain of the convolution operator \eqref{eq:convolution} (or equivalently of the linear system \eqref{eq:mainsyst_gen} with $x_0=0$) defined as
\begin{equation}
  ||\Sigma||_{L_p-L_p}:=\sup_{||w||_{L_p}=1}||\Sigma w||_{L_p},
\end{equation}
is finite if and only if $A$ is Hurwitz stable. In particular, when the system \eqref{eq:mainsyst_gen} is positive, then we have that
\begin{equation}
      ||\Sigma||_{L_p-L_p}=|| \widehat{\Sigma} (0)||_p
\end{equation}
for any $p\in\{1,2,\infty\}$. Note that it is often considered that inputs need to be nonnegative. However, it is immediate to see that for positive systems, the worst-case inputs are necessarily nonnegative since the impulse response is nonnegative as well. Therefore, imposing  this restriction is not necessary when defining the $L_p$-gain of a positive system. Also, it is interesting to note that the same definition also holds for externally positive systems, those systems for which the impulse response $h(t)$ is nonnegative at all times but which are not internally positive.

We finally have the following result that is due to \citep{Stoer:62}:
\begin{proposition}\label{prop:Stoer}
  Let $M\in\mathbb{R}_{\ge0}^{q\times q}$. Then, for all $p\in\{1,2,\infty\}$, we have that
  \begin{equation}
    \rho(M)=\inf_{D\in\mathbb{D}_{\succ0}^q}||DMD^{-1}||_p
  \end{equation}
  and the infimum is a minimum whenever $M$ is irreducible.
\end{proposition}

\section{Exact stability results for uncertain linear positive systems and interconnections of positive systems}\label{sec:general}

The aim of this section is to recall important results regarding the stability of uncertain linear positive systems and the stability of interconnections of linear positive systems. Both theoretical and computational results are provided, the latter being stated in terms of linear or semidefinite programs.

\subsection{Stability conditions for uncertain systems in LFT form}

We are interested here in the stability of the following uncertain systems in linear fractional form
 \begin{equation}\label{eq:uncertain}
  \begin{array}{rcl}
    \dot{x}(t)&=&Ax(t)+Ew(t)\\
    z(t)&=&Cx(t)+Fw(t)\\
    w(t)&=&\Delta z(t),\ \Delta\in\BDeltab{p}
  \end{array}
\end{equation}
where
\begin{equation}\label{eq:Bdelta}
   \BDeltab{p}:=\left\{\Delta\in\mathbb{C}^{q\times q}\left|\begin{array}{l}
   \Delta=\diag(\Delta_1,\ldots,\Delta_N),||\Delta_i||_p\le1\\
      \Delta_i\in\mathbb{C}^{q_i\times q_i},\ i=1,\ldots,N
   \end{array}\right.\right\}\\
\end{equation}
with $p\in\{1,2,\infty\}$ and $\textstyle q=\sum_{i=1}^Nq_i$. Associated with this uncertainty structure, we define the following set of constant $D$-scalings:
 \begin{equation}
  \mathcal{D}_\Delta:=\left\{D\in\mathbb{R}^{q\times q}\left|\begin{array}{c}
     D=\diag(d_1I_{q_1},\ldots,d_NI_{q_N})\\ d_i>0,\ i=1,\ldots,N
  \end{array}\right.\right\}.
\end{equation}
The role of the scalings is to capture the structure of the uncertainty set through the property that $\Delta D=D\Delta$ for all $\Delta\in\BDeltab{p}$ and all $D\in\mathcal{D}_\Delta$. Such scalings allow us to reduce the conservatism of the small-gain theorem and, in some certain cases, make the conservatism vanish completely. This latter effect will happen in the context of linear positive systems and will allow us to derive nonconservative stability results.

\subsubsection{General theoretical result}

We have the following result:
\begin{proposition}
Assume that $(A,E,C,F)$ is internally positive. Then, the following statements are equivalent:
\begin{enumerate}[(i)]
  \item The uncertain system \eqref{eq:uncertain} is asymptotically stable for all $\Delta\in\BDeltab{p}$.
  \item $A$ is Hurwitz stable and
  \begin{equation}
    \inf_{D\in\mathcal{D}_\Delta}||D\widehat{\Sigma} (0)D^{-1}||_p<1.
  \end{equation}
\end{enumerate}
Moreover, in the repeated scalar uncertainties (i.e. $\Delta_i=\delta_i I_{q_i}$, $\delta_i\in\mathbb{R}_{>0}$, $i=1,\ldots,N$), then the above statements are equivalent to
\begin{enumerate}[(i)]
\setcounter{enumi}{2}
\item $A$ is Hurwitz stable and
    \begin{equation}
    \rho(\widehat{\Sigma} (0))<1.
   \end{equation}
\end{enumerate}
\end{proposition}
\begin{proof}
 The equivalence (i) $\Leftrightarrow$ (ii) has been proved in \citep{Colombino:15} in the case $p=2$. The case $p=\infty$ has been proved, for instance, in \citep{Dahleh:95,Khammash:93}. Finally, the case $p=1$ is dual to the case $p=\infty$.  The equivalence (ii) $\Leftrightarrow$ (iii) follows from Proposition \ref{prop:Stoer}.
\end{proof}

\blue{Interestingly, in the case $p=2$, the internal stability condition on the system can be relaxed into the condition of positive domination \cite{Rantzer:12,Colombino:15}. When $p=1$ or $p=\infty$, the internal positivity condition can be substituted by an external positivity condition together with an assumption on the initial condition in order to preserve the positivity of the output (i.e. $x_0$ must be such that $Ce^{At}x_0\ge0$ for all $t\ge0$). Finally, since eventually positive systems can be used to efficiently represent externally positive systems, some of the results for internally positive systems are expected to remain true for these systems as well; see e.g. \cite{Sootla:15,Altafini:16}.}

\subsubsection{$L_1$ scaled small-gain theorem}

\blue{The following result can be seen as an extension of the $L_1$ results in \citep{Briat:11g,Ebihara:11,Briat:11h}:}
\begin{theorem}\label{th:ebihara}
Assume that $(A,E,C,F)$ is internally positive. Then, the following statements are equivalent:
\begin{enumerate}[(i)]
  \item The uncertain system \eqref{eq:uncertain} is asymptotically stable for all $\Delta\in\BDeltab{1}$.
  \item There exist a positive vector $\lambda\in\mathbb{R}_{>0}^{n}$ and a matrix $D\in\mathcal{D}_{\Delta}$ such that 
      \begin{equation}
      \begin{bmatrix}
        \lambda\\
        D\mathds{1}_q
      \end{bmatrix}^T\begin{bmatrix}
        A & E\\
        C & F-I_q
      \end{bmatrix}<0.
      \end{equation}
\end{enumerate}
\end{theorem}
\begin{proof}
 Following \citep{Briat:11h}, we have that $||D\widehat{\Sigma} (0)D^{-1}||_1<1$ if and only if there here exists a positive vector $\lambda\in\mathbb{R}^{n}_{>0}$ such that the inequalities
  \begin{equation}
    \begin{array}{rcl}
      \lambda^TA+\mathds{1}^T_qDC&<&0\\
      \lambda^TED^{-1}+\mathds{1}^TDFD^{-1}-\mathds{1}_{q}^T&<&0
    \end{array}
  \end{equation}
  hold. Right-multiplying the second inequality by $D\in\mathbb{D}_{\succ0}^n$ yields the result.
\end{proof}

\subsubsection{$L_2$ scaled small-gain theorem}

\blue{The following result, proved in \citep{Colombino:15}, is the positive version of the well-known $L_2$ scaled small-gain theorem \cite{Packard:93,Dullerud:00} and is based on the Kalman-Yakubovich-Popov (KYP) Lemma for positive systems \citep{Shorten:09,Tanaka:10}:}
\begin{theorem}\label{th:colombino}
Assume that $(A,E,C,F)$ is internally positive. Then, the following statements are equivalent:
\begin{enumerate}[(i)]
  \item The uncertain system \eqref{eq:uncertain} is asymptotically stable for all $\Delta\in\BDeltab{2}$.
  \item There exist matrices $P\in\mathbb{D}_{\succ0}^{n}$ and $D\in\mathcal{D}_{\Delta}$ such that
      \begin{equation}
       \begin{bmatrix}
       A^TP+PA & PE & C^TD\\
        \star & -D & F^TD\\
        \star & \star & -D
        \end{bmatrix}\prec0.
      \end{equation}
\end{enumerate}
\end{theorem}

\blue{Alternative formulations can also be obtained on the basis of the linear KYP lemma for positive systems proved in \cite{Rantzer:16}:
\begin{theorem}\label{th:rantzer}
Assume that $(A,E,C,F)$ is internally positive. Then, the following statements are equivalent:
\begin{enumerate}[(i)]
  \item The uncertain system \eqref{eq:uncertain} is asymptotically stable for all $\Delta\in\BDeltab{2}$.
  \item There exist $\lambda,\mu\in\mathbb{R}_{>0}^n$, $\nu\in\mathbb{R}_{>0}^q$ and $D\in\mathcal{D}_{\Delta}$ such that $A\lambda+E\nu<0$ and
  \begin{equation}\label{eq:dhsihdskdsjlkdjskdjldjsl}
    \begin{bmatrix}
      \lambda\\
      \nu
    \end{bmatrix}^T\begin{bmatrix}
      C^TDC & C^TDF\\
      F^TDC \quad& F^TDF-D
    \end{bmatrix}+\mu^T\begin{bmatrix}
      A \quad& E
    \end{bmatrix}<0
  \end{equation}
  hold.
\end{enumerate}
\end{theorem}
\begin{proof}
Applying  the linear version of the KYP Lemma from \cite{Rantzer:16} on the scaled system $D^{1/2}\widehat{\Sigma}(s) D^{-1/2}$ where  $D\in\mathcal{D}_{\Delta}$ yields the conditions $A\lambda+ED^{-1/2}\tilde{\nu}<0$ and
  \begin{equation}
    \begin{bmatrix}
      \lambda\\
      \tilde\nu^T
    \end{bmatrix}\begin{bmatrix}
      C^TDC & C^TDFD^{-1/2}\\
      D^{-1/2}F^TDC \quad& D^{-1/2}F^TDFD^{-1/2}-I
    \end{bmatrix}+\mu^T\begin{bmatrix}
      A \quad& ED^{-1/2}
    \end{bmatrix}<0
  \end{equation}
  for some positive vectors $\lambda,\nu$ and $\mu$. Note that these conditions are equivalent to saying that the LMI condition in Theorem \ref{th:colombino} holds (possibly with a different matrix $D$). The final result is then obtained by letting $\nu=D^{-1/2}\tilde\nu$ and by multiplying the above inequality from the right by the matrix $\diag(I,D^{1/2})$.
\end{proof}
Unfortunately, the condition \eqref{eq:dhsihdskdsjlkdjskdjldjsl} is not convex because of the product between $\lambda,\nu$ and $D$. In this regard, this condition may not be very convenient to work with for establishing the stability of the uncertain system \eqref{eq:uncertain} with $\Delta\in\BDeltab{2}$.}

\blue{Finally, it is also interesting to mention the following novel result based on a result in \cite{Naghnaeian:14}:
\begin{theorem}\label{th:naghnaeian}
Assume that $(A,E,C,F)$ is internally positive. Then, the following statements are equivalent:
\begin{enumerate}[(i)]
  \item The uncertain system \eqref{eq:uncertain} is asymptotically stable for all $\Delta\in\BDeltab{2}$.
  \item There exist $\zeta\in\mathbb{R}_{>0}^{n\times q}$ and $D\in\mathcal{D}_{\Delta}$ such that $\zeta^TA+DC<0$ and
  \begin{equation}
    \begin{bmatrix}
    -D \quad & \zeta^TE+DF\\
    \star & -D
    \end{bmatrix}\prec0
  \end{equation}
  hold.
\end{enumerate}
\end{theorem}
\begin{proof}
  To prove this one, we use a result of \cite{Naghnaeian:14} which exactly characterizes the $L_2$-gain of a linear positive system. By applying it to the scaled system $D^{1/2}\widehat{\Sigma}(s) D^{-1/2}$, we get the conditions $\tilde{\zeta}^TA+D^{1/2}E<0$ and
   \begin{equation}
    \begin{bmatrix}
    -I \quad & \tilde\zeta^TED^{-1/2}+D^{1/2}F^{-1/2}\\
    \star & -I
    \end{bmatrix}\prec0
  \end{equation}
  for some $\tilde\zeta\in\mathbb{R}_{>0}^{n\times q}$. A congruence transformation with respect to the matrix $\diag(D^{1/2},D^{1/2})$ and the change of variables $\zeta=\tilde{\zeta}D^{1/2}$ yield the result.
\end{proof}
}
Note that unlike the condition in Theorem \ref{th:rantzer}, the condition in Theorem \ref{th:naghnaeian} is still convex but not linear. Its complexity is also higher than the complexity of the condition in Theorem \ref{th:colombino}.

\subsubsection{$L_\infty$ scaled small-gain theorem}

The following result is the ``$L_\infty$ counterpart" of Theorem \ref{th:colombino} which can also be seen as an extension of the results in \citep{Briat:11h} and a version of the scaled small-gain theorem in the $L_\infty$-sense:
\begin{theorem}\label{th:briat}
Assume that $(A,E,C,F)$ is internally positive. Then, the following statements are equivalent:
\begin{enumerate}[(i)]
  \item The uncertain system \eqref{eq:uncertain} is asymptotically stable for all $\Delta\in\BDeltab{\infty}$.
  \item There exist a positive vector $\lambda\in\mathbb{R}_{>0}^{n}$ and a matrix $D\in\mathcal{D}_{\Delta}$ such that 
            \begin{equation}
            \begin{bmatrix}
              A & E\\
              C & F-I_q
            \end{bmatrix}\begin{bmatrix}
              \lambda\\
              D\mathds{1}_q
            \end{bmatrix}<0.
      \end{equation}
\end{enumerate}
\end{theorem}
\begin{proof}
 Following \citep{Briat:11h}, we have that $||D^{-1}\widehat{\Sigma} (0)D||_\infty<1$ if and only if there here exists a positive vector $\lambda\in\mathbb{R}^{n}_{>0}$ such that the inequalities
  \begin{equation}
    \begin{array}{rcl}
      A\lambda+ED\mathds{1}_q&<&0\\
      D^{-1}C\lambda+D^{-1}FD\mathds{1}_q-\mathds{1}_{q}&<&0
    \end{array}
  \end{equation}
  hold. Left-multiplying the second inequality by $D$ yields the result.
\end{proof}

\subsection{Stability of interconnections using Integral Linear Constraints}

In the current setup, we are interested in the analysis of interconnections of the form
\begin{equation}\label{eq:G1G2}
  \begin{array}{rcl}
    u_2&=&G_1u_1+d_2\\
    u_1&=&G_2u_2+d_1
  \end{array}
\end{equation}
where $G_1:L_1\mapsto L_1$ and $G_2:L_1\mapsto L_1$ are bounded linear positive time-invariant operators with transfer functions $\widehat{G}_1(s)$ and $\widehat{G}_2(s)$. Note that since the operators are positive, then we have that $\widehat{G}_1(0)\ge0$ and $\widehat{G}_2(0)\ge0$. The signals $u_1,u_2$ are the loop signals and $d_1,d_2$ are the exogenous signals which are all assumed to have dimensions that are compatible with the operators $G_1$ and $G_2$.
%
The next result is a simplified, specialized and extended version of the ones in \cite{Briat:15cdc} \blue{where ILC/separation types results have been formulated. Note that the statement (iv) has also been reported in \cite{Ebihara:11} whereas the statement (v) seems novel.}
\blue{\begin{theorem}\label{th:ILC}
The following statements are equivalent:
\begin{enumerate}[(i)]
  \item The interconnection \eqref{eq:G1G2} is well-posed, positive and stable\footnote{For some additional details about these concepts see \cite{Briat:15cdc}}.
  \item We have that $\rho(\widehat{G}_1(0)\widehat{G}_2(0))<1$.
  \item There exist some vectors $\pi_1\in\mathbb{R}_{>0}^{m}$ and $\pi_2\in\mathbb{R}^{p}$ such that the conditions
   \begin{equation}
   \pi_1^T+\pi_2^T\widehat{G}_2(0)\ge0\ \textnormal{and}\ \pi_1^T\widehat{G}_1(0)+\pi_2^T<0
  \end{equation}
  hold.
  \end{enumerate}
  Moreover, when the internally positive systems $G_1$ and $G_2$ can be represented in terms of the rational transfer functions $\widehat{G}_i(s)=C_i(sI-A_i)^{-1}E_i+F_i$ with $A_i\in\mathbb{R}^{n_i\times n_i}$, $A_i$ Metzler, $C_i\in\mathbb{R}_{\ge0}^{s_i\times n_i}$, $E_i\in\mathbb{R}_{\ge0}^{n_i\times r_i}$ and $F_i\in\mathbb{R}_{\ge0}^{s_i\times r_i}$, $i=1,2$, $s_1=r_2$, $s_2=r_1$, then the above statements are also equivalent to
  \begin{enumerate}[(i)]
  \setcounter{enumi}{3}
  \item There exist some vectors $\lambda_i\in\mathbb{R}^{n_i}$, $\mu_1\in\mathbb{R}_{>0}^{s_1}$ and $\mu_2\in\mathbb{R}_{>0}^{s_2}$ such that the conditions
      \begin{equation}\label{ea:djszkdjsldjlsdjslkjdl}
        \begin{array}{rcl}
          \lambda_1^TA_1+\mu_1^TC_1&<&0\\
          \lambda_1^TE_1+\mu_1^TF_1-\mu_2^T&<&0\\
          \lambda_2^TA_2+\mu_2^TC_2&<&0\\
          \lambda_2^TE_2+\mu_2^TF_2-\mu_1^T&<&0
        \end{array}
      \end{equation}
      hold.
  \item There exist some matrices $P_i\in\mathbb{D}_{\succ0}^{n_i}$ and $Q_i\in\mathbb{S}_{\succ0}^{s_i}$, $i=1,2$ such that the conditions
  \begin{equation}
    \begin{bmatrix}
      A_1^TP_1+P_1A_1 \quad& P_1E_1 \quad& C_1^TQ_1\\
      \star & -Q_2 & F_1^TQ_1\\
      \star  & \star & -Q_1
    \end{bmatrix}\prec0
  \end{equation}
  and
  \begin{equation}\label{ea:djszkdjsldjlsdjslkjdl2}
    \begin{bmatrix}
      A_2^TP_2+P_2A_2 \quad& P_2E_2 \quad& C_2^TQ_2\\
      \star & -Q_1 & F_2^TQ_2\\
      \star  & \star & -Q_2
    \end{bmatrix}\prec0
  \end{equation}
  hold.
\end{enumerate}
\end{theorem}
\begin{proof}
The proof that (i) and (ii) are equivalent follows from \cite{Briat:15cdc}. The equivalence between (ii) and (iii) follows from standard algebraic manipulations and the fact that $\rho(\widehat{G}_1(0)\widehat{G}_2(0))<1$ if and only if there exists a positive vector $\lambda$ of compatible dimensions such that $\lambda^T(\widehat{G}_1(0)\widehat{G}_2(0)-I)<0$. The equivalence between (iii) and (iv) follows from standard algebraic manipulations together with the change of variables $\mu_1=\pi_1$ and $\mu_2=-\pi_2$. This statement has also been proven in \cite{Ebihara:11}. To prove the equivalence between (ii) and (v), first note that from \cite{Rantzer:16}, the LMIs are equivalent to saying that
\begin{equation}
  \begin{bmatrix}
    \widehat{G}_1(0)\\
    I
  \end{bmatrix}^T\begin{bmatrix}
    Q_1 \quad & 0\\
    0 & -Q_2
  \end{bmatrix} \begin{bmatrix}
    \widehat{G}_1(0)\\
    I
  \end{bmatrix}=\widehat{G}_1(0)^TQ_1\widehat{G}_1(0)-Q_2\prec0
\end{equation}
and
\begin{equation}
  \begin{bmatrix}
    \widehat{G}_2(0)\\
    I
  \end{bmatrix}^T\begin{bmatrix}
    Q_2 \quad & 0\\
    0 & -Q_1
  \end{bmatrix} \begin{bmatrix}
    \widehat{G}_2(0)\\
    I
  \end{bmatrix}=\widehat{G}_2(0)^TQ_2\widehat{G}_2(0)-Q_1\prec0.
\end{equation}
These two inequalities together imply that $\widehat{G}_2(0)^T\widehat{G}_1(0)^TQ_1\widehat{G}_1(0)\widehat{G}_2(0)-Q_1\prec0$ or $\widehat{G}_1(0)^T\widehat{G}_2(0)^TQ_2\widehat{G}_2(0)\widehat{G}_1(0)-Q_2\prec0$, which are equivalent to saying that (ii) holds. The converse can be proven by first noting that (ii) is equivalent to saying that there exists a matrix $R\in\mathbb{R}_{\succ0}^{s_1}$ such that  $\widehat{G}_2(0)^T\widehat{G}_1(0)^TR\widehat{G}_1(0)\widehat{G}_2(0)-R\prec0$. This implies the existence of a sufficiently small $S\in\mathbb{S}_{\succ0}^{s_1}$ such that $\widehat{G}_2(0)^T\widehat{G}_1(0)^T(R+S)\widehat{G}_1(0)\widehat{G}_2(0)-R\prec0$. Letting now $Q_2=\widehat{G}_1(0)^T(R+S)\widehat{G}_1(0)$ and $R=Q_1$, we get that $\widehat{G}_2(0)^TQ_2\widehat{G}_2(0)-Q_1\prec0$ and $\widehat{G}_1(0)^TQ_1\widehat{G}_1(0)-Q_2=-\widehat{G}_1(0)^TS\widehat{G}_1(0)\preceq0$. In order to prove that opening the inequality sign is not restrictive, it is enough to note that if $\widehat{G}_2(0)^TQ_2\widehat{G}_2(0)-Q_1\prec0$ and $\widehat{G}_1(0)^TQ_1\widehat{G}_1(0)-Q_2\preceq0$ then we can always positively perturb the value $Q_2$ so that we have  $\widehat{G}_2(0)^TQ_2\widehat{G}_2(0)-Q_1\prec0$ and $\widehat{G}_1(0)^TQ_1\widehat{G}_1(0)-Q_2\prec0$. The proof is completed.
\end{proof}
Note that the difference with the results in \cite{Briat:11g,Briat:11g,Ebihara:11,Colombino:15} is that the above result may deal with more general systems than LTI systems with state-space realization but can deal with any type of linear positive bounded operators such as bounded integral operators. This will be useful for dealing with systems with distributed delays. It is also interesting to mention that the above result is not a small-gain result but a separation result \cite{Safonov:78,Iwasaki:98a,Ebihara:11,Briat:15cdc}. Note, finally, that even though the matrices $Q_1$ and $Q_2$ are defined to be general symmetric matrices, one can be chosen to be diagonal without losing the necessity of the result. However, it is unclear whether they can be both chosen as diagonal without losing necessity.

\begin{remark}\label{remark:dszkdls}
  It is interesting to see that the two last inequalities in \eqref{ea:djszkdjsldjlsdjslkjdl} can be substituted by $\mu_2^TG_2(0)-\mu_1^T<0$ while the LMI \eqref{ea:djszkdjsldjlsdjslkjdl2} can be substituted by $G_2(0)^TQ_2G_2(0)-Q_1\prec0$. This will be useful when analyzing systems with distributed delays in Section \ref{sec:distributed_main}.
\end{remark}
}

\section{Stability and performance of linear positive systems with discrete-delays}\label{sec:discrete_main}


\subsection{Stability analysis -- the constant delay case}\label{sec:discreteTI}

We start with the following result \citep{GuKC:03,Briat:book1}:
\begin{proposition}
Let $p\in\{1,2,\infty\}$. Then, the time-varying delay operator
\begin{equation}
  \begin{array}{rclcl}
  \mathscr{D}_h&:&L_p&\mapsto &L_p,\\
  &&w(t)&\mapsto&w(t-h),\ h\ge 0
  \end{array}
\end{equation}
has unitary $L_1$-, $L_2$- and $L_\infty$-gains.%
\end{proposition}
Note that by virtue of the Riesz-Thorin interpolation theorem \cite{Hormander:85}, we immediately get that all the $L_p$-gains, for $p=1,2,\ldots,\infty$, of the above delay operator are equal to 1.

\begin{proposition}\label{prop:TDS1}
The linear time-delay system
\begin{equation}\label{eq:TDS1}
  \dot{x}(t)=A_0x(t)+\sum_{i=1}^NA_ix(t-h_i)
  \end{equation}
  coincides with the uncertain system \eqref{eq:uncertain} where $A=A_0$, $E=[A_1\ \ldots\ A_N]$, $C=\mathds{1}_N\otimes I_n$, $F=0$ and
  \begin{equation*}
    \Delta\in\boldsymbol{\Delta_d}:=\left\{\diag_{i=1}^N(e^{-sh_i}I_n):\ h\ge0,\ \Re(s)\ge0\right\}.
  \end{equation*}
  %
\end{proposition}
\begin{proof}
  The proof follows from direct substitutions.
\end{proof}
It is known that the system \eqref{eq:TDS1} is positive if and only if $A_0$ is Metzler and $A_i$ is nonnegative for all $i=1,\ldots,N$; see e.g. \cite{Haddad:04}. We can now state the main result that unifies the results in  \citep{Haddad:04,Briat:11h,Aleksandrov:16}:
\begin{theorem}\label{th:TDS1}
 Assume that the system \eqref{eq:TDS1} is positive. Then, the following statements are equivalent:
  \begin{enumerate}[(i)]
    \item The system \eqref{eq:TDS1} is asymptotically stable.
    \item The matrix $\textstyle\sum_{i=0}^NA_i$ is Hurwitz stable.
    \item There exists a vector $v\in\mathbb{R}^n_{>0}$ such that $v^T\left(\textstyle\sum_{i=0}^NA_i\right)<0$.
    \item There exist matrices $P,Q_i\in\mathbb{D}^n_{\succ0}$, $i=1,\ldots,N$, such that the Riccati inequality
    \begin{equation}\label{eq:kdsodlsmdkkdsmd}
      A_0^TP+PA_0+\sum_{i=1}^N\left(Q_i+PA_iQ_i^{-1}A_i^TP\right)\prec0
    \end{equation}
    holds.
    \item $A_0$ is Hurwitz stable and $\textstyle-(\sum_{i=1}^NA_i)A_0^{-1}$ is Schur stable.
  \end{enumerate}
\end{theorem}
\begin{proof}
The equivalence between the two first statements follows from the application of Theorem \ref{th:briat} on the system \eqref{eq:TDS1}. Indeed, by applying the linear programming conditions of Theorem  \ref{th:briat} we get that the system is asymptotically stable if and only there exist vectors $\lambda\in\mathbb{R}_{>0}^{n}$ and $\nu\in\mathbb{R}_{>0}^{Nn}$ such that $A_0\lambda+E\nu<0$ and $\mathds{1}_N\otimes\lambda-\nu<0$. This is equivalent to say that we have $A_0\lambda+E(\mathds{1}_N\otimes\lambda)=(\textstyle\sum_{i=0}^NA_i)\lambda<0$, which is equivalent to the first statement of the result. The equivalence between the statements (ii) and (iii) follows from standard results on the stability of Metzler matrices whereas the equivalence between (i) and (iv) follows from Theorem \ref{th:colombino} where we have set $D=\textstyle\diag_{i=1}^N(Q_i)$ where $Q_i\in\mathbb{D}_{\succ0}^{q_i}$. Note that in this case, the uncertainty structure is diagonal and hence $\mathcal{D}_\Delta$ is the set of diagonal matrices with positive diagonal entries. Finally, statement (v) is obtained from Proposition \ref{prop:Stoer}.
\end{proof}

\blue{It is interesting to provide few remarks regarding the above result. First of all, we recover the property that the stability of a linear positive system with constant and discrete-delays does not depend on the delay values and hence stability is a delay-independent property.} Secondly, the second statement of the above result provides an answer to a particular version of the problem stated in \cite[Problem 1.11]{Blondel:04bk} by E. Verriest on the Riccati stability of linear time-delay systems with a single discrete and constant delay. This problem is about finding conditions on the matrices $A_0$ and $A_1$ (i.e. in the case $N=1$) for which there exist $P,Q_1\in\mathbb{S}^n_{\succ0}$ such that \eqref{eq:kdsodlsmdkkdsmd} holds. The above result  provides a solution to this problem for the particular cases of linear positive and positively dominated systems with delays. Note that the positive systems case has also been solved in \citep{Aleksandrov:16,Mason:12} using different approaches. These results have since been extended to some other classes of systems in \cite{Aleksandrov:16b}.

The advantages of the proposed approach over the previously mentioned ones are its simplicity and its flexibility. Indeed, while the proofs of these results in the above works involve some very technical developments, the proposed approach allows to retrieve the same results through a very simple application of the scaled small-gain theorems. Moreover, the approach can be easily extended to other types of uncertainties, to performance analysis and to design purposes. Finally, the statement (iv) in the above result is also interestingly as it generalizes the frequency-sweeping results of \cite[Section 2.3]{GuKC:03} where this condition is stated as necessary for a linear system with delay to be stable. For linear positive systems, this condition is also sufficient and this result can be interpreted as a consequence of the fact that the maximum value of the spectral radius is always attained at the zero frequency.

\blue{\subsection{Performance analysis -- the constant delay case}

Let us now consider the system
\begin{equation}\label{eq:TDS1_input}
\begin{array}{rcl}
  \dot{x}(t)&=&A_0x(t)+\sum_{i=1}^NA_ix(t-h_i)+E_uu(t)\\
  y(t)&=&C_0x(t)+\sum_{i=1}^NC_ix(t-h_i)+F_uu(t)\\
  \end{array}
  \end{equation}
where $u\in\mathbb{R}^{n_u}$ and $y\in\mathbb{R}^{n_y}$ are the input and the output, respectively. It is known that the above system is positive if and only if $A_0$ is Metzler and $C,E_u,F_u,A_i,C_i$ are nonnegative for all $i=1,\ldots,N$. We then have the following result:
\begin{theorem}\label{prop:perf_CDD}
  Assume that the system \eqref{eq:TDS1_input} is positive. Then, the following statements hold:
  \begin{enumerate}[(i)]
    \item The system \eqref{eq:TDS1_input} is asymptotically stable and has an $L_1$-gain smaller than $\gamma_1$ if and only if there exists a vector $\lambda\in\mathbb{R}^n_{>0}$ such that
    \begin{equation}
        \begin{bmatrix}
          \lambda\\
          \mathds{1}_{n_y}
          \end{bmatrix}^T\begin{bmatrix}
            \textstyle \sum_{i=0}^NA_i &\quad& E_u\\
           \textstyle  \sum_{i=0}^NC_i &\quad& F_u
          \end{bmatrix}<\begin{bmatrix}
            0\\
            \gamma_1\mathds{1}_{n_u}
          \end{bmatrix}^T.
    \end{equation}
      \item The system \eqref{eq:TDS1_input} is asymptotically stable and has an $L_2$-gain smaller than $\gamma_2$ if and only if there exists a matrix $P,Q_i\in\mathbb{D}^n_{\succ0}$, $i=1,\ldots,N$, such that
          \begin{equation}
        \begin{bmatrix}
          A_0^TP+PA_0+\sum_{i=1}^NQ_i &\quad \row_{i=1}^N(PA_i) &\quad PE_u & C_0^T\\
          \star & -\diag_{i=1}^N(Q_i) & 0 &  \col_{i=1}^N(C_i^T)\\
          \star  &\star& -\gamma_2 I & F_u^T\\
                    \star  &\star& \star & -\gamma_2 I
        \end{bmatrix}\prec0.
    \end{equation}
    \item The system \eqref{eq:TDS1_input} is  asymptotically stable and has an $L_\infty$-gain smaller than $\gamma_\infty$ if and only if there exists a vector $\lambda\in\mathbb{R}^n_{>0}$ such that
    \begin{equation}
    \begin{bmatrix}
             \textstyle \sum_{i=0}^NA_i &\quad& E_u\\
           \textstyle  \sum_{i=0}^NC_i &\quad& F_u
          \end{bmatrix}\begin{bmatrix}
          \lambda\\
          \mathds{1}_{n_u}
          \end{bmatrix}<\begin{bmatrix}
            0\\
            \gamma_\infty\mathds{1}_{n_y}
          \end{bmatrix}.
    \end{equation}
  \end{enumerate}
\end{theorem}
\blue{\begin{proof}
  The proof of this result is based on the fact that we can also rewrite the performance characterization problem as a robust stability problem by setting $u=My$ where $M\ge0$ is a full-block matrix such that $||M||_p=\gamma_p^{-1}$ (or, equivalently, $||M||_p\le\gamma_p^{-1}$). Since for linear positive systems scaled-small gain results are non-conservative, then the stability of the interconnection is equivalent to the fact that the $L_p$-gain of the transfer $u\mapsto y$ is at most (or smaller than) $\gamma_p$. Note, however, that the interconnection result applies to square uncertainties, a condition that is violated when  $n_u\ne n_y$ (recall that the consideration of $D$ scalings requires the uncertainty to be square). This issue can be easily resolved by suitably augmenting the vector $y$ or the vector $u$ and appropriately padding the associated matrices with zeros.

We only prove this result in the $L_1$-gain case, the others are analogous, and we also assume, without loss of generality, that $n_u=n_y$. To this aim, let us consider the interconnection
  \begin{equation}
  \begin{array}{rcl}
        \begin{bmatrix}
      \dot{x}\\
      \hline
      z\\
      \hline
      y
    \end{bmatrix}&=&\begin{bmatrix}
      A_0 & \vline & \row_{i=1}^N(A_i) & \vline & E_u\\
      \hline
      \mathds{1}_N\otimes I_n & \vline & 0 & \vline & 0\\
      \hline
      C_0 & \vline & \row_{i=1}^N(C_i) & \vline & F_u
    \end{bmatrix} \begin{bmatrix}
      x\\
      \hline
      w\\
      \hline
      u
    \end{bmatrix}\\
    w&=&\Delta z,\ \Delta\in\boldsymbol{\Delta_d}\\
    u&=&My,\ ||M||_1\le\gamma_1^{-1}.
  \end{array}
  \end{equation}
  Applying now Theorem \ref{th:ebihara} on the above system with the extended uncertainty $\Delta_e=\diag(\Delta,M)$ with extended scaling $D_e=\diag(D,\epsilon I_{n_u})$,  $D\in\mathbb{D}^{Nn}_{\succ0}$, $\epsilon>0$, yields
  \begin{equation}
    \begin{bmatrix}
     \tilde \lambda\\
      \hline
     \tilde \mu\\
      \hline
      \epsilon\mathds{1}^T_{n_y}
    \end{bmatrix}^T\begin{bmatrix}
      A_0 & \vline & \row_{i=1}^N(A_i) & \vline & E_u\\
      \hline
      \mathds{1}_N\otimes I_n & \vline & 0 & \vline & 0\\
      \hline
      C_0 & \vline & \row_{i=1}^N(C_i) & \vline & F_u
    \end{bmatrix}<  \begin{bmatrix}
      0\\
      \hline
      \tilde\mu\\
      \hline
      \epsilon\gamma_1\mathds{1}_{n_u}
    \end{bmatrix}^T.
  \end{equation}
  Solving for $\tilde\mu$ as in the proof of Theorem \ref{th:TDS1}, dividing everything by $\epsilon$ and using the change of variables $\lambda=\tilde\lambda/\epsilon$ yield the result.
\end{proof}}
}

\subsection{Stability analysis -- the time-varying delay case}\label{sec:discreteTV}

We extend here the previous results to the case of time-varying discrete-delays. Interestingly, this class of delays includes, as a special case, \blue{scale}-delays \cite{Verriest:99b,Briat:book1}. A remark will be made in this regard. Note also that this case is different from the previous one as the gain of the delay operator will be different depending on the considered norm. This is formalized in the following result \citep{GuKC:03,Briat:11j,Briat:book1}:
\begin{proposition}
Let $p\in\{1,2,\infty\}$. Then, the time-varying delay operator
\begin{equation}
  \begin{array}{rclcl}
  \mathscr{T}_h&:&L_p&\mapsto &L_p\\
  &&w(t)&\mapsto&w(t-h(t)),\ h(t)\ge 0
  \end{array}
\end{equation}
has
\begin{itemize}
  \item an $L_p$-gain equal to $(1-\eta)^{-1/p}$ where $\dot{h}(t)\le\eta<1$, $p\in\mathbb{Z}_{>0}$ and
  \item an $L_\infty$-gain equal to 1.
\end{itemize}
\end{proposition}
\blue{The above result clearly shows that, unlike for the constant delay-operator, the value of the gain of the time-varying delay operator depends on the considered norm. Note also that while the $L_p$-gains, $p\in\mathbb{Z}_{>0}$, depend on the maximum rate of change of the delay, the $L_\infty$-gain does not, a property that makes it appropriate for the consideration of fast-varying delays. This fact is not surprising since delay operators do not change the amplitude of the input signal but may change the value of its integral by appropriately distorting time; see e.g. \cite{Kao:07}. Yet, it is possible to determine a finite $L_2$-gain when the rate of change of the delay may exceed one (see \cite{Shustin:07}) and an analogous result for the $L_1$-gain seems to be still missing. Finally, it is interesting to note that the $L_\infty$-gain is smaller than the others since $1<(1-\eta)^{-1/p}$ for any $p\in\mathbb{Z}_{>0}$ and, in this regard, this gain may be more interesting to use than the other. This claim will be supported by the main results of the section but, before proving them, we need to state the following preliminary result:}
%


\begin{proposition}\label{prop:TDS3}
The linear delay system
\begin{equation}\label{eq:TDS3}
  \dot{x}(t)=A_0x(t)+\sum_{i=1}^NA_ix(t-h_i(t))
\end{equation}
  coincides with the uncertain system \eqref{eq:uncertain} where $A=A_0$, $E=[A_1\ \ldots\ A_N]$, $C=\mathds{1}_N\otimes I_n$, $F=0$ and
  \begin{equation*}
    \Delta\in\left\{\diag_{i=1}^N(I_n\otimes\mathscr{T}_{h_i}):\ h_i:\mathbb{R}_{\ge0}\to\mathbb{R}_{\ge0},\ i=1,\ldots,N\right\}.
  \end{equation*}
\end{proposition}

We can now state the main result that unifies the results in  \citep{Haddad:04,AitRami:09,Briat:11h,Shen:15}
\begin{theorem}\label{th:TDS3}
    Assume that the system \eqref{eq:TDS3} is internally positive and that $t-h_i(t)\to\infty$ as $t\to\infty$ for all $i=1,\ldots,N$. Then, the following statements are equivalent:
  \begin{enumerate}[(i)]
   \item The system \eqref{eq:TDS3} is asymptotically stable.
    \item The matrix $\textstyle\sum_{i=0}^NA_i$ is Hurwitz stable.
    \item There exist vectors $\lambda,\mu_i\in\mathbb{R}^n_{>0}$, $i=1,\ldots,N$ such that
    \begin{equation}
      \begin{bmatrix}
        A_0 & \row_{i=1}^N(A_i)\\
        \mathds{1}_{N}\otimes I_n & -I_{nN}\\
      \end{bmatrix}\begin{bmatrix}
        \lambda\\
        \col_{i=1}^N(\mu_i)
      \end{bmatrix}<0.
    \end{equation}
  \end{enumerate}
\end{theorem}
\begin{proof}
First note that since the delays are time-varying, then the second statement is necessary for the stability of the system \eqref{eq:TDS3} and, hence, (i) implies (ii). Note also that (ii) and (iii) are equivalent from Theorem \ref{th:TDS1}. Finally, using Theorem \ref{th:briat} and the fact that the $L_\infty$-gain of the time-varying delay operator is equal to one, we can conclude that (iii) implies (i), which completes the proof.
\end{proof}
\blue{
  Interestingly, we can see that a necessary and sufficient condition for a linear positive time-delay system with constant or time-varying discrete-delays is that $\textstyle\sum_{i=0}^NA_i$ be Hurwitz stable. As a consequence, the existence of diagonal solutions $P,Q_i\in\mathbb{D}^n_{\succ0}$, $i=1,\ldots,N$ to the Riccati inequality
    \begin{equation}\label{eq:kdsodlsmdkkdsmd2}
      A_0^TP+PA_0+\sum_{i=1}^N\left(Q_i+PA_iQ_i^{-1}A_i^TP\right)\prec0
    \end{equation}
    is also a necessary and sufficient condition for the stability of linear positive systems with time-varying delays. This result is, however, rather surprising if we take into account the fact that if we were applying the $L_2$-scaled small-gain result (i.e. Theorem \ref{th:colombino}) on the system \eqref{eq:TDS3}, we would get the following Riccati inequality
 \begin{equation}\label{eq:lkdsmdq}
      A_0^TP+PA_0+\sum_{i=1}^N\left(Q_i+(1-\eta_i)^{-1}PA_iQ_i^{-1}A_i^TP\right)\prec0
 \end{equation}
 where the $\eta_i$'s are such that $\dot{h}_i(t)\le\eta_i<1$, $i=1,\ldots,N$,  for all $t\ge0$. Therefore, if \eqref{eq:lkdsmdq} is feasible, then so is \eqref{eq:kdsodlsmdkkdsmd2}, but the converse is not true in general. In this regard, we would not be able to predict that the stability of \eqref{eq:TDS1} is equivalent to the stability of \eqref{eq:TDS3} using Theorem  \ref{th:colombino}, nor even Theorem \ref{th:ebihara}. 

 More generally, the stability condition for $N=1$ obtained in the $L_p$, $p\in\{1,2,\infty\}$, framework is equivalent to saying that
 \begin{equation}
   \rho(-A_0^{-1}A_1)<(1-\eta)^{1/p}
 \end{equation}
 which indicates that the $L_1$-based result is more conservative than the $L_2$-based result which is, in turn, more conservative than the $L_\infty$-based result. Even though this remark seems contradictory with the fact that scaled small-gain results are nonconservative, it is actually not since we are considering now with time-varying operators and also because stability in the $L_p$-sense is analyzed. Indeed, we have the set of nonnegative continuous functions converging to 0 with finite $L_1$-norm is strictly included in the set of nonnegative continuous functions converging to 0 with finite $L_2$-norm, which is itself strictly included in the set of nonnegative continuous functions converging to 0 with finite $L_\infty$-norm. The fact that the $L_\infty$-based result gives the weakest stability condition demonstrates its relevance.
}

\blue{\subsection{Performance analysis -- the time-varying delay case}

Let us now consider the system
\begin{equation}\label{eq:TDS2_input}
\begin{array}{rcl}
  \dot{x}(t)&=&A_0x(t)+\sum_{i=1}^NA_ix(t-h_i(t))+E_uu(t)\\
  y(t)&=&C_0x(t)+\sum_{i=1}^NC_ix(t-h_i(t))+F_uu(t)\\
  \end{array}
  \end{equation}
where $u\in\mathbb{R}^{n_u}$ and $y\in\mathbb{R}^{n_y}$ are the input and the output, respectively. As for system \eqref{eq:TDS1_input}, the above system is internally positive if and only if $A_0$ is Metzler and $C,E_u,F_u,A_i,C_i$ are nonnegative for all $i=1,\ldots,N$.

We then have the following result:
\begin{theorem}
  Assume that the system \eqref{eq:TDS2_input} is positive. Then, the following statements hold:
  \begin{enumerate}[(i)]
    \item  Assume that the delays are differentiable and such that $\dot{h}_i(t)\le\eta_i<1$, $i=1,\ldots,N$. Then, the system \eqref{eq:TDS2_input} is asymptotically stable and has an $L_1$-gain smaller than $\gamma_1$ if there exists a vector $\lambda\in\mathbb{R}^n_{>0}$ such that
    \begin{equation}
        \begin{bmatrix}
          \lambda\\
          \mathds{1}_{n_y}
          \end{bmatrix}^T\begin{bmatrix}
            \textstyle A_0+\sum_{i=1}^N(1-\eta_i)^{-1}A_i &\quad& E_u\\
           \textstyle  C_0+\sum_{i=1}^N(1-\eta_i)^{-1}C_i &\quad& F_u
          \end{bmatrix}<\begin{bmatrix}
            0\\
            \gamma_1\mathds{1}_{n_u}
          \end{bmatrix}^T.
    \end{equation}
      \item  Assume that the delays are differentiable and such that $\dot{h}_i(t)\le\eta_i<1$, $i=1,\ldots,N$. Then, the system \eqref{eq:TDS2_input} is asymptotically stable and has an $L_2$-gain smaller than $\gamma_2$ if there exist matrices $P,Q_i\in\mathbb{D}^n_{\succ0}$, $i=1,\ldots,N$, such that
          \begin{equation}
        \begin{bmatrix}
          A_0^TP+PA_0+\sum_{i=1}^NQ_i &\quad \row_{i=1}^N(PA_i) &\quad PE_u & C_0^T\\
          \star & -\diag_{i=1}^N((1-\eta_i)Q_i) & 0 &  \col_{i=1}^N(C_i^T)\\
          \star  &\star& -\gamma_2 I & F_u^T\\
                    \star  &\star& \star & -\gamma_2 I
        \end{bmatrix}\prec0.
    \end{equation}
    %
    \item Assume that the delays are such that  $t-h_i(t)\to\infty$ as $t\to\infty$ for all $i=1,\ldots,N$. Then,  the system \eqref{eq:TDS2_input} is  asymptotically stable and has an $L_\infty$-gain smaller than $\gamma_\infty$ if and only if there exists a vector $\lambda\in\mathbb{R}^n_{>0}$ such that
    \begin{equation}
    \begin{bmatrix}
            \textstyle A_0+\sum_{i=1}^NA_i &\quad& E_u\\
           \textstyle  C_0+\sum_{i=1}^NC_i &\quad& F_u
          \end{bmatrix}\begin{bmatrix}
          \lambda\\
          \mathds{1}_{n_u}
          \end{bmatrix}<\begin{bmatrix}
            0\\
            \gamma_\infty\mathds{1}_{n_y}
          \end{bmatrix}.
    \end{equation}
  \end{enumerate}
\end{theorem}
\begin{proof}
  The proof is similar to that of Theorem \ref{prop:perf_CDD} and is thus omitted.
\end{proof}}
\blue{\begin{remark}
  It is worth mentioning that the condition in statement (i) is necessary and sufficient while it is unclear, at the moment, whether necessity holds for the conditions in the statements (ii) and (iii).
\end{remark}}

\blue{

\section{Stability and performance of linear positive delay-difference equations}\label{sec:difference}

\subsection{Stability analysis}

We consider in this section, the case of delay-difference equations \cite{Avellar:80,Damak:14,Damak:15,Melchor:16,Melchor:13,Shen:15} of the form
\begin{equation}\label{eq:delaydifference}
  x(t)=\sum_{i=1}^NA_ix(t-h_i)
\end{equation}
where the delays $h_i$ are such that $h_i>0$ for all $i=1,\ldots,N$. Clearly, the system is positive if and only if the matrices $A_i$ are nonnegative for all $i=1,\ldots,N$. Note that this system can be rewritten as the interconnection
\begin{equation}\label{eq:delaydifference2}
  \begin{array}{rcl}
    x(t)&=&Ew(t)\\
    z(t)&=&Cx(t)
  \end{array}
\end{equation}
where $E=\begin{bmatrix}
  A_1 & \ldots & A_n
\end{bmatrix}$, $C=\mathds{1}_N\otimes I_n$ and $w=\Delta z$, $\Delta\in\boldsymbol{\Delta_d}$. We then have the following result:
\begin{theorem}\label{th:strongly}
  Assume that the system \eqref{eq:delaydifference} is positive. Then, the following statements are equivalent:
  \begin{enumerate}[(i)]
    \item The delay-difference equation is asymptotically stable (or strongly-stable).
    \item $\rho\left(\sum_{i=1}^NA_ie^{-j\omega_i}\right)<1$ for all $\omega_i\in\mathbb{R}$, $i=1,\ldots,N$.
    \item $\rho\left(\sum_{i=1}^NA_i\right)<1$.
    \item There exists a $\mu\in\mathbb{R}^{nN}_{>0}$ such that
    \begin{equation}
      \mu^T\left(-I_{nN}+(\mathds{1}_N\otimes I_n)\row_{i=1}^N(A_i)\right)<0
    \end{equation}
    holds.
    \item There exists a $\mu\in\mathbb{R}^{nN}_{>0}$ such that
    \begin{equation}
      \mu^T\left(\sum_{i=1}^NA_i-I_{n}\right)<0
    \end{equation}
    holds.
    \item There exist diagonal matrices $Q_i\in\mathbb{D}^n_{\succ0}$ such that the LMI
    \begin{equation}\label{eq:jdszkdjsakldjsldlj}
      \begin{bmatrix}
        -\diag_{i=1}^N(Q_i) & \star\\
        \diag_{i=1}^N(Q_i)(\mathds{1}_N\otimes I_n)\row_{i=1}^N(A_i) & -\diag_{i=1}^N(Q_i)
      \end{bmatrix}   \prec0
    \end{equation}
    holds.
  \end{enumerate}
\end{theorem}
\begin{proof}
  The equivalence between the two first statements has been proved in  \cite{Avellar:80}. The equivalence between (iii), (iv) and (v) follows from simple algebraic manipulations and the theory of nonnegative matrices. The equivalence between (iii) and (iv) follows from the fact that \eqref{eq:jdszkdjsakldjsldlj} is equivalent to
  \begin{equation}
    -\diag_{i=1}^N(Q_i)+((\mathds{1}_N\otimes I_n)\row_{i=1}^N(A_i) )^T\diag_{i=1}^N(Q_i)(\mathds{1}_N\otimes I_n)\row_{i=1}^N(A_i) \prec0
  \end{equation}
  which is, in turn, equivalently, that $\textstyle\rho[(\mathds{1}_N\otimes I_n)\row_{i=1}^N(A_i)]=\rho\left(\sum_{i=1}^NA_i\right)<1$. The equivalence between (ii) and (iv) can be proved using a scaled-small gain argument on the system \eqref{eq:delaydifference2}.
\end{proof}

\begin{remark}
When the delays are time-varying, the condition in (iv) remains valid as long as the delays satisfy the condition $t-h_i(t)\to\infty$ as $t\to\infty$ for all $i=1,\ldots,N$. As a result, the stability of the delay-difference equation does not depend on the value of the delays nor on their time-varying nature as long a $L_\infty$-gain result is considered but will depend on the rate of variation of the delays when $L_1$-and $L_2$-gain results are considered. See the discussion below Theorem \ref{th:TDS3} for additional details.
\end{remark}

\subsection{Performance analysis}

Let us consider here the following delay-difference equation
\begin{equation}\label{eq:delaydifference_input}
\begin{array}{rcl}
    x(t)&=&\sum_{i=1}^NA_ix(t-h_i)+E_uu(t)\\
    y(t)&=&\sum_{i=1}^NC_ix(t-h_i)+F_uu(t)
\end{array}
\end{equation}
where the delays $h_i$ are such that $h_i>0$ for all $i=1,\ldots,N$. Clearly, the system is positive if and only if the matrices $A_i,E_u,C_i,F_u$ are nonnegative for all $i=1,\ldots,N$. We then have the following result:
\begin{theorem}
  Assume that the system \eqref{eq:delaydifference_input} is positive. Then, we have the following results:
  \begin{enumerate}[(i)]
    \item The system \eqref{eq:delaydifference_input} is asymptotically stable and has an $L_1$-gain smaller than $\gamma_1$ if and only if there exists a $\mu\in\mathbb{R}^{nN}_{>0}$ such that the condition
    \begin{equation}
    \begin{bmatrix}
       \mu\\
       \mathds{1}_{n_y}
    \end{bmatrix}^T\begin{bmatrix}
      (\mathds{1}_N\otimes I_n)\row_{i=1}^N(A_i)\quad & \col_{i=1}^N(E_u)\\
      \row_{i=1}^N(C_i) & F_u
    \end{bmatrix}<  \begin{bmatrix}
       \mu\\
       \gamma_1\mathds{1}_{n_u}
    \end{bmatrix}^T
    \end{equation}
    holds.
    \item The system \eqref{eq:delaydifference_input} is asymptotically stable and has an $L_2$-gain smaller than $\gamma_2$ if and only if there exists matrices $Q_i\in\mathbb{D}_{\succ0}^{n}$ such that the condition
  \begin{equation}
      \begin{bmatrix}
        -\diag_{i=1}^N(Q_i) & \star & \star & \star\\
        0 & -\gamma_2 I_{n_u} & \star & \star\\
       \col_{i=1}^N(Q_i)\row_{i=1}^N(A_i)\quad & \col_{i=1}^N(Q_iE_u) & -\diag_{i=1}^N(Q_i) & \star\\
      \row_{i=1}^N(C_i) & F_u   &  0 & -\gamma_2 I_{n_y}
      \end{bmatrix}   \prec0
    \end{equation}
    holds.
   \item  The system \eqref{eq:delaydifference_input} is asymptotically stable and has an $L_\infty$-gain smaller than $\gamma_\infty$ if and only if there exists a $\mu\in\mathbb{R}^{nN}_{>0}$ such that the condition
    \begin{equation}
   \begin{bmatrix}
      (\mathds{1}_N\otimes I_n)\row_{i=1}^N(A_i)\quad & \col_{i=1}^N(E_u)\\
      \row_{i=1}^N(C_i) & F_u
    \end{bmatrix} \begin{bmatrix}
       \mu\\
       \mathds{1}_{n_u}
    \end{bmatrix}<  \begin{bmatrix}
       \mu\\
       \gamma_\infty\mathds{1}_{n_y}
    \end{bmatrix}
    \end{equation}
    holds.
  \end{enumerate}
\end{theorem}
\begin{proof}
The proof is based on reformulating the system \eqref{eq:delaydifference_input} into an LTI system interconnected with some delay operators and applying scaled small-gain results.
\end{proof}

\begin{remark}
Interestingly, the $L_\infty$ result remains the same when the delays are time-varying and such that $t-h_i(t)\to\infty$ as $t\to\infty$ for all $i=1,\ldots,N$. As for linear positive systems with discrete-delays, the conditions for the $L_1$- and the $L_2$-gains will be different (i.e. they will depend on the rate of variation of the delays and become sufficient only).
\end{remark}
}

\blue{\section{Stability and performance of linear positive coupled differential-difference equations with delays}\label{sec:coupled}

We consider here linear positive coupled differential-difference equations with delays which can be seen as an extension of the systems \eqref{eq:TDS1} and \eqref{eq:delaydifference}.  Such systems have been, for instance, studied in \cite{Hale:77,Niculescu:00a,Niculescu:00b,Pepe:03,Niculescu:06,Verriest:09d} and in the references therein. In particular, the case of linear positive  coupled differential-difference equations with single time-varying delay has been studied in \cite{Shen:15b} where necessary and sufficient conditions for their positivity and stability were obtained. We prove here that these results can be retrieved and extended to the case of multiple delays and to performance analysis using very simple scaled-small gain arguments.

\subsection{Stability analysis}

Let us start with the following preliminary result:
\begin{proposition}\label{prop:TDS1c}
The linear coupled differential-difference system with constant discrete delays
\begin{equation}\label{eq:TDS1c}
\begin{array}{rcl}
   \dot{x}_1(t)&=&A_0x_1(t)+\sum_{i=1}^NA_ix_2(t-h_i)\\
   x_2(t)&=&C_0x_1(t)+\sum_{i=1}^NC_ix_2(t-h_i)
\end{array}
  \end{equation}
  coincides with the uncertain system \eqref{eq:uncertain} where $A=A_0$, $E=\row_{i=1}^N(A_i)$, $C=\mathds{1}_N\otimes C_0$, $F=\mathds{1}_N\otimes\row_{i=1}^N(C_i)$ and
  \begin{equation*}
    \Delta\in\boldsymbol{\Delta_d}:=\left\{\diag_{i=1}^N(e^{-sh_i}I_n):\ h\ge0,\ \Re(s)\ge0\right\}.
  \end{equation*}
  %
\end{proposition}
\begin{proof}
  The proof follows from direct substitutions.
\end{proof}

It was proven in \cite{Shen:15b} that the system \eqref{eq:TDS1c} is positive if and only if  $A_0$ is Metzler and $A_i,C_0,C_i$ are nonnegative for all $i=1,\ldots,N$. Note, moreover, that this result is obvious from the description \eqref{eq:uncertain}  with the matrices defined in Proposition \ref{prop:TDS1c}.

We can now state the main result that extends the results in \citep{Shen:15b} with the difference that constant delays are considered. Note, however, that the result still holds in the case of time-varying delays in the same way as in Section \ref{sec:discreteTV}.
\begin{theorem}\label{th:TDS1c}
 Assume that the system \eqref{eq:TDS1c} is positive. Then, the following statements are equivalent:
  \begin{enumerate}[(i)]
    \item The system \eqref{eq:TDS1c} is asymptotically stable.
    \item The matrix
    \begin{equation}
      \mathcal{M}:=\begin{bmatrix}
        A_0 & \quad&\row_{i=1}^N(A_i)\\
        \mathds{1}_N\otimes C_0 & \quad&\mathds{1}_N\otimes\row_{i=1}^N(C_i)-I_{nN}
      \end{bmatrix}
    \end{equation}
    is Hurwitz stable.
    \item There exists a vector $v\in\mathbb{R}^{n(N+1)}_{>0}$ such that $v^T\mathcal{M}<0$.
    \item There exist some matrices $P\in\mathbb{D}^n_{\succ0}$ and $Q\in\mathbb{D}^{Nn}_{\succ0}$ such that the generalized Riccati inequality
\begin{equation}\label{eq:kdsodlsmdkkdsmd22}
      A_0^TP+PA_0+\begin{bmatrix}
        P\row_{i=1}^N(A_i) & \vline & \mathds{1}_N^T\otimes C_0^T
      \end{bmatrix}\begin{bmatrix}
        -Q & \vline & (\mathds{1}_N^T\otimes\col_{i=1}^N(C_i^T))Q\\
        \hline
        Q(\mathds{1}_N\otimes\row_{i=1}^N(C_i))\quad & \vline & -Q
      \end{bmatrix}^{-1}\begin{bmatrix}
        \col_{i=1}^N(A_i^T)P\\ \hline \mathds{1}_N\otimes C_0
      \end{bmatrix}\prec0
    \end{equation}
    holds.
    \item The matrices $\mathds{1}_N\otimes\row_{i=1}^N(C_i)-I$ and
    \begin{equation}
      A_0-\begin{bmatrix}
        A_1 & A_2 & \ldots & A_N
      \end{bmatrix}( \mathds{1}_N\otimes\row_{i=1}^N(C_i)-I)^{-1}(\mathds{1}_N\otimes C_0)
    \end{equation}
    are Hurwitz stable.
    %
    \item The matrices $A_0$ and
    \begin{equation}
       \mathds{1}_N\otimes\row_{i=1}^N(C_i)-I-(\mathds{1}_N\otimes C_0)A_0^{-1}\begin{bmatrix}
        A_1 & A_2 & \ldots & A_N
      \end{bmatrix}
    \end{equation}
    are Hurwitz stable.
    \item $A_0$ is Hurwitz stable and $\textstyle\sum_{i=1}^N(-C_0A_0^{-1}A_i+C_i)$ is Schur stable.
  \end{enumerate}
\end{theorem}
\begin{proof}
The equivalence between the three first statements follows from the application of Theorem \ref{th:ebihara} on the input-output formulation of the system \eqref{eq:TDS1c}. The equivalence with the statements (iv) and (vii) comes from Theorem \ref{th:colombino} and the fact that, for two matrices $A,B$ of appropriate dimensions, we have that $\rho(AB)=\rho(BA)$. The equivalence between (ii), (v) and (vi) follows from Lemma 1 in \cite{Ebihara:11} or Lemma 7.2 in \cite{Briat:14c}.
\end{proof}}
\blue{As for linear systems with discrete-delays, the stability of linear positive coupled differential-difference equations with delays is equivalent to that one of the same system with all the delays set to 0.}

\blue{\subsection{Performance analysis}

Let us now consider the linear coupled differential-difference system with time-varying discrete delays
\begin{equation}\label{eq:TDS1c_input}
\begin{array}{rcl}
   \dot{x}_1(t)&=&A_0x_1(t)+\sum_{i=1}^NA_ix_2(t-h_i(t))+E_1u(t)\\
    x_2(t)&=&C_0x_1(t)+\sum_{i=1}^NC_ix_2(t-h_i(t))+E_2u(t)\\
   y(t)&=&C_{y0}x(t)+\sum_{i=1}^NC_{yi}x_2(t-h_i(t))+F_uu(t)
  \end{array}
  \end{equation}
where $u\in\mathbb{R}^{n_u}$ and $y\in\mathbb{R}^{n_y}$ are the input and the output, respectively. It is immediate to see that this system is internally positive if and only if $A_0$ is Metzler and $A_i,C_0,C_i,E_1,E_2,C_{y0},C_{yi},F_u$ are nonnegative for all $i=1,\ldots,N$. We then have the following result:
\begin{theorem}\label{th:TDS1c}
 Assume that the system \eqref{eq:TDS1c_input} is positive. Then, the following statements hold:
    \begin{enumerate}[(i)]
    \item Assume that the delays are differentiable and such that $\dot{h}_i(t)\le\eta_i<1$, $i=1,\ldots,N$. Then, the system \eqref{eq:TDS1c_input} has an $L_1$-gain smaller than $\gamma_1$ if there exists some vectors $\lambda\in\mathbb{R}^n_{>0}$ and $\mu\in\mathbb{R}^{Nn}_{>0}$ such that
        \begin{equation}
        \begin{bmatrix}
          \lambda\\
          \mu\\
          \mathds{1}_{n_y}\\
          \end{bmatrix}^T\begin{bmatrix}
            A_0 & \quad & \row_{i=1}^N((1-\eta_i)^{-1}A_i) & \quad & E_1\\
           \col_{i=1}^N(C_0) & \quad & \col_{i=1}^N(I_n)\row_{i=1}^N((1-\eta_i)^{-1}C_i) &  \quad & \col_{i=1}^N(E_2)\\
          C_{y_0} & \quad & \row_{i=1}^N((1-\eta_i)^{-1}C_i)  & \quad & F_u
      \end{bmatrix}<\begin{bmatrix}
            0\\
            \mu\\
            \gamma_1\mathds{1}_{n_u}
          \end{bmatrix}^T.
    \end{equation}
\item Assume that the delays are differentiable and such that $\dot{h}_i(t)\le\eta_i<1$, $i=1,\ldots,N$. Then, the system \eqref{eq:TDS1c_input} has an $L_2$-gain smaller than $\gamma_2$ if there exist matrices $P,Q_i\in\mathbb{D}^n_{\succ0}$, $i=1,\ldots,N$, such that
          \begin{equation}
        \begin{bmatrix}
          A_0^TP+PA_0+\sum_{i=1}^NQ_i &\quad \row_{i=1}^N(PA_i) &\quad PE_u & \row_{i=1}^N(C_0^T)Q & C_{y0}^T\\
          \star & -\diag_{i=1}^N((1-\eta_i)Q_i) & 0 &  (\mathds{1}_N\otimes\row_{i=1}^N(C_i))^T & \col_{i=1}^N(C_{yi}^T)\\
          \star  &\star& -\gamma_2 I & \row_{i=1}^N(E_2^T) & F_u^T\\
          \star  &\star& \star & - Q & 0\\
                    \star  &\star& \star & \star  & -\gamma_2 I
        \end{bmatrix}\prec0.
    \end{equation}
    %
    \item Assume that the delays are such that $t-h_i(t)\to\infty$ as $t\to\infty$, $i=1,\ldots,N$. Then, the system \eqref{eq:TDS1c_input} has an $L_\infty$-gain smaller than $\gamma_\infty$ if and only if there exists some vectors $\lambda\in\mathbb{R}^n_{>0}$ and $\mu\in\mathbb{R}^{Nn}_{>0}$ such that
   \begin{equation}
       \begin{bmatrix}
            A_0 & \quad & \row_{i=1}^N(A_i) & \quad & E_1\\
           \col_{i=1}^N(C_0) & \quad & \col_{i=1}^N(I_n)\row_{i=1}^N(C_i) &  \quad & \col_{i=1}^N(E_2)\\
          C_{y_0} & \quad & \row_{i=1}^N(C_i)  & \quad & F_u
      \end{bmatrix} \begin{bmatrix}
          \lambda\\
          \mu\\
          \mathds{1}_{n_u}\\
          \end{bmatrix}<\begin{bmatrix}
            0\\
            \mu\\
            \gamma_\infty\mathds{1}_{n_y}
          \end{bmatrix}.
    \end{equation}
  \end{enumerate}
\end{theorem}}
\blue{As for systems with time-varying discrete delays, the last statement states a necessary and sufficient condition whereas it is unclear whether the conditions in the two first ones are also necessary. Note also that in the $L_1$ and $L_\infty$ cases, the vector $\mu$ can be eliminated from the conditions by explicitly solving it. However, the benefit of the current conditions is that they are linear in the matrices of the system, thereby allowing for immediate extensions to uncertain matrices and to design purposes.}

\section{Stability and performance of linear positive systems with distributed-delays}\label{sec:distributed_main}

\subsection{Stability analysis -- the constant  kernel case}\label{sec:distributed}

We have the following result which be proven using standard manipulations:
\begin{proposition}
The time-varying distributed-delay operator
\begin{equation}
  \begin{array}{rclcl}
  \mathscr{U}_h&:&L_p&\mapsto &L_p,\ p=2,\infty\\
  &&w(t)&\mapsto&\dfrac{1}{\bar{h}}\int_{t-h(t)}^tw(\theta)d\theta,\ h(t)\in[0,\bar{h}]
  \end{array}
\end{equation}
has unitary  $L_1$-, $L_2$- and $L_\infty$-gains.
\end{proposition}

\begin{proposition}\label{prop:distributed}
The time-delay system
\begin{equation}\label{eq:TDS4}
  \dot{x}(t)=A_0x(t)+\sum_{i=1}^NA_i\int_{t-h_i(t)}^tx(s)ds
\end{equation}
coincides with the uncertain system \eqref{eq:uncertain} where $A=A_0$, $E=[A_1\ \ldots\ A_N]$, $C=\mathds{1}_N\otimes I_n$, $F=0$ and
  \begin{equation*}
    \Delta\in\left\{\diag_{i=1}^N(I_n\otimes\mathscr{U}_{h_i}):\ h_i:\mathbb{R}_{\ge0}\to[0,\bar{h}_i],\ i=1,\ldots,N\right\}.
  \end{equation*}
  Moreover, it is internally positive if and only if the matrix $A_0$ is Metzler and the matrices $A_i$, $i=1,\ldots,N$, are nonnegative.
\end{proposition}

We then have the following result:
\begin{theorem}\label{th:TDS2}
    Assume that the system \eqref{eq:TDS4} is internally positive. Then, the following statements are equivalent:
  \begin{enumerate}[(i)]
   \item The system \eqref{eq:TDS4} is asymptotically stable.
    \item $\textstyle A_0+\sum_{i=1}^N\bar{h}_iA_i$ is Hurwitz stable.
    \item There exists a $v\in\mathbb{R}_{>0}^n$ such that $\textstyle v^T(A+\sum_{i=1}^N\bar{h}_iA_i)<0$.
    \item $A$ is Hurwitz stable and $-(\textstyle\sum_{i=1}^N\bar{h}_i A_i)A_0^{-1}$ is Schur stable.
    \item There exist matrices $P,Q_i\in\mathbb{D}^n_{\succ0}$, $i=1,\ldots,N$, such that the Riccati inequality
    \begin{equation}
    A_0^TP+PA_0+\sum_{i=1}^N(Q_i+\bar{h}_i^2PA_iQ_i^{-1}A_i^TP)\prec0
    \end{equation}
    holds.
  \end{enumerate}
  Moreover, when $N=1$, then the above statements are also equivalent to
  \begin{enumerate}[(i)]
    \setcounter{enumi}{5}
        \item $A_0$ is Hurwitz stable and $\bar{h}_1<\dfrac{1}{\rho(A_0^{-1}A_1)}$.
  \end{enumerate}
\end{theorem}
\begin{proof}
 This result is proved exactly in the same way as Theorem \ref{th:TDS1}. The last statement can be straightforwardly shown to be equivalent to (iv).
\end{proof}%

\subsection{Performance analysis -- the constant  kernel case}

Let us now consider the system
\begin{equation}\label{eq:TDS4_input}
\begin{array}{rcl}
   \dot{x}(t)&=&A_0x(t)+\sum_{i=1}^NA_i\int_{t-h_i(t)}^tx(s)ds+E_uu(t)\\
  y(t)&=&C_0x(t)+\sum_{i=1}^NC_i\int_{t-h_i(t)}^tx(s)ds+F_uu(t)\\
  \end{array}
  \end{equation}
  where $u\in\mathbb{R}^{n_u}$ and $y\in\mathbb{R}^{n_y}$ are the input and the output, respectively. As for system \eqref{eq:TDS1_input}, the above system is internally positive if and only if $A_0$ is Metzler and $C,E_u,F_u,A_i,C_i$ are nonnegative for all $i=1,\ldots,N$. We have the following result:
\blue{\begin{theorem}
  Assume that the system \eqref{eq:TDS4_input} is internally positive. Then, the following statements hold:
  \begin{enumerate}[(i)]
    \item The system \eqref{eq:TDS4_input} is asymptotically stable and has an $L_1$-gain smaller than $\gamma_1$ if and only there exists a vector $\lambda\in\mathbb{R}^n_{>0}$ such that
    \begin{equation}
        \begin{bmatrix}
          \lambda\\
          \mathds{1}_{n_y}
          \end{bmatrix}^T\begin{bmatrix}
            \textstyle A_0+\sum_{i=1}^N\bar{h}_iA_i &\quad& E_u\\
            \textstyle  C_0+\sum_{i=1}^N\bar{h}_iC_i &\quad& F_u
          \end{bmatrix}<\begin{bmatrix}
            0\\
            \gamma_1\mathds{1}_{n_u}
          \end{bmatrix}^T.
    \end{equation}
      \item The system \eqref{eq:TDS4_input} is asymptotically stable and has an $L_2$-gain smaller than $\gamma_2$ if and only if there exist matrices $P,Q_i\in\mathbb{D}^n_{\succ0}$, $i=1,\ldots,N$, such that
             \begin{equation}
        \begin{bmatrix}
          A_0^TP+PA_0+\sum_{i=1}^NQ_i &\quad \row_{i=1}^N(\bar{h}_iPA_i) &\quad PE_u & C_0^T\\
          \star & -\diag_{i=1}^N(Q_i) & 0 & \col_{i=1}^N(\bar{h}_iC_i^T)\\
          \star & \star & -\gamma_2 I & F_u^T\\
          \star  &\star& \star & -\gamma_2 I
        \end{bmatrix}\prec0
    \end{equation}
    %
    \item The system \eqref{eq:TDS4_input} is  asymptotically stable and has an $L_\infty$-gain smaller than $\gamma_\infty$ if and only if there exists a vector $\lambda\in\mathbb{R}^n_{>0}$ such that
    \begin{equation}
    \begin{bmatrix}
                \textstyle A_0+\sum_{i=1}^N\bar{h}_iA_i &\quad& E_u\\
            \textstyle  C_0+\sum_{i=1}^N\bar{h}_iC_i &\quad& F_u
          \end{bmatrix}\begin{bmatrix}
          \lambda\\
          \mathds{1}_{n_u}
          \end{bmatrix}<\begin{bmatrix}
            0\\
            \gamma_\infty\mathds{1}_{n_y}
          \end{bmatrix}.
    \end{equation}
  \end{enumerate}
\end{theorem}}
\begin{proof}
  The proof follows from the same lines as the proof of Theorem \ref{prop:perf_CDD}.
\end{proof}

\subsection{Stability analysis -- the non-constant  kernel case}\label{sec:distributed2}

To prove this result, we need to consider the result based on integral linear constraints (Theorem \ref{th:ILC}) since results based on gains would be conservative. Let us first consider the following result:
\begin{proposition}
The distributed-delay operator
\begin{equation}\label{eq:dkslsldkklmdklmsdmksmdklm}
  \begin{array}{rclcl}
  \mathscr{V}_{B,h}&:&L_p&\mapsto &L_p\\
  &&w(t)&\mapsto&\int_{-h}^0B(\theta)w(t+\theta)d\theta,\ 0\le h\le \bar{h}
  \end{array}
\end{equation}
is nonnegative if and only if $B(\theta)\ge0$ for all $\theta\in[-\bar h,0]$, $h\in[0,\bar{h}]$. Moreover, its transfer function is given by
\begin{equation}
  \widehat{\mathscr{V}_{B,h}}(s)=\int_{-h}^0B(\theta)e^{s\theta}d\theta
\end{equation}
and for any $0\le h\le\bar h$, we have that $\widehat{\mathscr{V}_{B,h}}(0)\le\widehat{\mathscr{V}_{B,\bar h}}(0)$.
\end{proposition}
\begin{proof}
Let $z(t)=\mathscr{V}_{B,h}(w)(t)$. To prove the nonnegativity property condition, assume that $B_{ij}(s)<0$ for some $s\in[-\bar h,0]$. Then, it is immediate to see that we can pick a $w$ such that one of the component of the output is negative. Hence, the result follows. The transfer function can be computed as follows
\begin{equation}
  \begin{array}{lcl}
    \widehat{z}(s)&=&\int_0^\infty z(t)e^{-st}dt\\
                               &=& \int_0^\infty \int_{-h}^0B(\theta)w(t+\theta)e^{-st}d\theta dt\\
                               &=&  \int_{-h}^0B(\theta)\left(\int_0^\infty w(t+\theta)e^{-st}dt\right)d\theta \\
                               &=&\left(\int_{-h}^0B(\theta)e^{s\theta}d\theta\right)\widehat{w}(s)\\
                               &=&  \widehat{\mathscr{V}_{B,h}}(s)\widehat{w}(s).
  \end{array}
\end{equation}
\end{proof}

\begin{remark}
  Note that the delay can be made infinite as in \citep{Solomon:13} as long  as $B(\theta)$ is integrable on $(-\infty,0]$.
\end{remark}

\begin{proposition}
   The linear system with time-varying distributed-delays
    \begin{equation}\label{eq:dkqsodmklskodsfklmsd}
 \dot{x}(t)=A_0x(t)+\sum_{i=1}^N\int_{-h_i(t)}^0A_i(\theta)x(t+\theta)d\theta
\end{equation}
with $h_i(t)\in[0,\bar{h}_i]$, $i=1,\ldots,N$,  is positive if and only if $A_0$ is Metzler and $A_i(\theta)\ge0$ for all $\theta_i\in[-\bar{h}_i,0]$ and all $i=1,\ldots,N$.
\end{proposition}

The following result demonstrates that the stability of linear positive systems with distributed delays does not depend on the nature of the delay (i.e. whether it is time-varying or time-invariant) and only depends on the delay upper-bound \cite{Shen:14}:
\begin{theorem}
Assume that the system \eqref{eq:dkqsodmklskodsfklmsd} is positive. Then, the statements are equivalent:
  \begin{enumerate}[(i)]
    \item The linear positive system \eqref{eq:dkqsodmklskodsfklmsd} with time-varying distributed-delays is asymptotically stable for any $h_i:\mathbb{R}_{\ge0}\mapsto[0,\bar h_i]$.
\item The linear positive system with constant distributed-delays
    \begin{equation}\label{eq:TDS4}
  \dot{\bar x}(t)=A_0\bar x(t)+\int_{-\bar h_i}^0A_i(\theta)\bar x(t+\theta)d\theta
\end{equation}
is asymptotically stable.
  \end{enumerate}
\end{theorem}
\begin{proof}
  \blue{  The proof can be found in \cite{Shen:14} and is thus omitted here.}
\end{proof}

\begin{proposition}\label{prop:distributed}
The time-delay system \eqref{eq:TDS4}
coincides with the interconnection \eqref{eq:G1G2} with
\begin{equation}
  \begin{array}{lcl}
    \widehat{G}_1(s)&=&C(sI-A_0)^{-1}E\\
    \widehat{G}_2(s)&=&\diag_{i=1}^N(\widehat{\mathscr{V}_{A_i,h}}(s))
  \end{array}
\end{equation}
where $E=\mathds{1}_N^T\otimes I_n$ and $C=\mathds{1}_N\otimes I_n$.
\end{proposition}

%
%
%

%
We then have the following result:
\begin{theorem}\label{th:TDS4}
    Define $\textstyle \bar{A}_i:=\int_{-\bar h_i}^0A_i(\theta)d\theta$. Then, the following statements are equivalent:
  \begin{enumerate}[(i)]
   \item The system \eqref{eq:TDS4} is asymptotically stable.
    \item $\textstyle A_0+\sum_{i=1}^N\bar{A}_i$ is Hurwitz stable.
    \item There exists a $v\in\mathbb{R}_{>0}^n$ such that $\textstyle v^T(A_0+\sum_{i=1}^N\bar A_i)<0$.
    \item $A_0$ is Hurwitz stable and $\textstyle-(\sum_{i=1}^N\bar A_i)A_0^{-1}$ is Schur stable.
    \item There exist matrices $P,Q_i\in\mathbb{D}^n_{\succ0}$, $i=1,\ldots,N$, such that the (diagonal) Riccati inequality
    \begin{equation}
    A_0^TP+PA_0+\sum_{i=1}^N(Q_i+P\bar A_iQ_i^{-1}\bar A_i^TP)\prec0
    \end{equation}
    holds.
  \end{enumerate}
\end{theorem}
\begin{proof}
Applying Theorem \ref{th:ILC} yields that the interconnection is well-posed, positive and stable if and only if $\rho(\widehat{G}_1(0)\widehat{G}_2(0))<1$ or, equivalently, if and only if  $\rho(\widehat{G}_2(0)\widehat{G}_1(0))<1$. Expanding the product yields
\begin{equation}
  \widehat{G}_2(0)\widehat{G}_1(0)=-(\mathds{1}_N\otimes I_n)A_0^{-1}\begin{bmatrix}
    \bar{A}_1 & \ldots & \bar{A}_N
  \end{bmatrix}
\end{equation}
and hence the spectral radius condition is equivalent to saying that the above nonnegative matrix is Schur stable. Since $\rho(AB)=\rho(BA)$ for any matrices $A,B$ of compatible dimensions, then we get the result of statement (iv). By exploiting the similarity of the conditions with those of Theorem \ref{th:TDS1}, the other statements directly follow.
\end{proof}

The last statement in the above result is the distributed-delay analogue of the Riccati inequality for discrete delays and does not seem to have been reported elsewhere in the literature. It is also interesting to point out that the results would have been completely different if scaled small-gain theorems would have been considered. Indeed, such results would have considered the $L_p$-gain of the operator \eqref{eq:dkslsldkklmdklmsdmksmdklm} which are given by $||\widehat{\mathscr{V}_{A_i,h}}(0)||_p$ for $p\in\{1,2,\infty\}$. It is clear that, in such a case, the obtained results would have been conservative.

\blue{\subsection{Performance analysis -- the non-constant  kernel case}

Let us consider here is the following system with distributed delays
    \begin{equation}\label{eq:TDS5_input}
    \begin{array}{rcl}
        \dot{x}(t)&=&A_0x(t)+\sum_{i=1}^N\int_{-\bar h_i}^0A_i(\theta)x(t+\theta)d\theta+E_uu(t)\\
        y(t)&=&C_0x(t)+\sum_{i=1}^N\int_{-\bar h_i}^0C_i(\theta)x(t+\theta)d\theta+F_uu(t)
    \end{array}
\end{equation}
 where $u\in\mathbb{R}^{n_u}$ and $y\in\mathbb{R}^{n_y}$ are the input and the output, respectively. As for system \eqref{eq:dkqsodmklskodsfklmsd}, the above system is internally positive if and only if $A_0$ is Metzler, $C,E_u,F_u$ are nonnegative for all $i=1,\ldots,N$ and the functions $A_i(\cdot),C_i(\cdot)$ are nonnegative on their domain.

\begin{theorem}
  Assume that the system \eqref{eq:TDS5_input} is internally positive and let us define the matrices
  \begin{equation}
    \bar{A}_i:=\int_{-\bar{h}_i}^0A_i(s)ds\ \textnormal{and}\     \bar{C}_i:=\int_{-\bar{h}_i}^0C_i(s)ds.
  \end{equation}
Then, the following statements hold:
  \begin{enumerate}[(i)]
    \item The system \eqref{eq:TDS5_input} is asymptotically stable and has an $L_1$-gain smaller than $\gamma_1$ if and only there exists a vector $\lambda\in\mathbb{R}^n_{>0}$ such that
    \begin{equation}
        \begin{bmatrix}
          \lambda\\
          \mathds{1}_{n_y}
          \end{bmatrix}^T\begin{bmatrix}
            \textstyle A_0+\sum_{i=1}^N\bar{A}_i &\quad& E_u\\
            \textstyle  C_0+\sum_{i=1}^N\bar{C}_i &\quad& F_u
          \end{bmatrix}<\begin{bmatrix}
            0\\
            \gamma_1\mathds{1}_{n_u}
          \end{bmatrix}^T.
    \end{equation}
      \item The system \eqref{eq:TDS5_input} is asymptotically stable and has an $L_2$-gain smaller than $\gamma_2$ if and only if there exist matrices $P,Q_i\in\mathbb{D}^n_{\succ0}$, $R_i\in\mathbb{D}_{\succ0}^{n_y}$, $i=1,\ldots,N$, such that
             \begin{equation}
  \begin{bmatrix}
    A_0^TP+PA_0+\sum_{i=1}^N\bar{A}_i^TQ_i\bar{A}_i+\sum_{i=1}^N\bar{C}_i^TR_i\bar{C}_i  & \mathds{1}_N^T\otimes P & 0& PE_u & C_0^T\\
   \star &  -\diag_{i=1}^N(Q_i) &  0 & 0 &0\\
    \star &  \star &  -\diag_{i=1}^N(R_i) & 0 & \mathds{1}_N\otimes I_{n_y}\\
     \star &  \star&  \star & -\gamma_2I_{n_u} & F_u^T\\
     \star &  \star&  \star & \star & -\gamma_2I_{n_y}
  \end{bmatrix}\prec0.
\end{equation}
    %
    \item The system \eqref{eq:TDS5_input} is  asymptotically stable and has an $L_\infty$-gain smaller than $\gamma_\infty$ if and only if there exists a vector $\lambda\in\mathbb{R}^n_{>0}$ such that
    \begin{equation}
    \begin{bmatrix}
                \textstyle A_0+\sum_{i=1}^N\bar{A}_i &\quad& E_u\\
            \textstyle  C_0+\sum_{i=1}^N\bar{C}_i&\quad& F_u
          \end{bmatrix}\begin{bmatrix}
          \lambda\\
          \mathds{1}_{n_u}
          \end{bmatrix}<\begin{bmatrix}
            0\\
            \gamma_\infty\mathds{1}_{n_y}
          \end{bmatrix}.
    \end{equation}
  \end{enumerate}
\end{theorem}
\begin{proof}
\noindent\textbf{Proof of the statement (i).} From the Theorem \ref{th:ILC} and Remark \ref{remark:dszkdls}, this is equivalent to saying that there exist $\lambda\in\mathbb{R}_{>0}^n$, $\mu_1,\mu_2\in\mathbb{R}_{>0}^{Nn+Nn_y}$ such that
\begin{equation}
  \begin{bmatrix}
    \lambda\\
    \hline
    \mu_1\\
    \hline
    \mathds{1}_{n_y}
  \end{bmatrix}^T\begin{bmatrix}
    A_0 & \vline & \mathds{1}_N^T\otimes I_n & 0 &\vline& E_u\\
    \hline
    \mathds{1}_N\otimes I_n & \vline & 0 & 0 & \vline & 0\\
    \mathds{1}_N\otimes I_{n_y} & \vline & 0 & 0 & \vline & 0\\
    \hline
    C_0 & \vline & 0 & \mathds{1}_N^T\otimes I_{n_y} &\vline& F_u
    \end{bmatrix}<\begin{bmatrix}
    0\\
    \hline
    \mu_2\\
    \hline
    \gamma_1\mathds{1}_{n_u}
  \end{bmatrix}^T
\end{equation}
together with $\mu_2^T\mathcal{D}-\mu_1^T<0$ where $\mathcal{D}=\diag\left(\diag_{i=1}^N\left(\bar{A}_i\right),\diag_{i=1}^N\left(\bar{C}_i\right)\right)$. This is equivalent to the inequalities
\begin{equation}
  \begin{array}{rcl}
    \lambda^TA_0+\mu_1^T\begin{bmatrix}
         \mathds{1}_N\otimes I_n\\
    \mathds{1}_N\otimes I_{n_y}
    \end{bmatrix}+\mathds{1}_{n_y}^TC_0&<&0\\
      \lambda^T\begin{bmatrix}
         \mathds{1}_N^T\otimes I_n & 0_{n\times Nn_y}\end{bmatrix}+\mathds{1}_{n_y}^T\begin{bmatrix}
           0_{n_y\times Nn} & \mathds{1}_N^T\otimes I_{n_y}
         \end{bmatrix}-\mu_2^T&<&0\\
         \lambda^T E_u+\mathds{1}_{n_y}^TF_u-\gamma_1\mathds{1}_{n_u}^T&<&0.
  \end{array}
\end{equation}
Using the fact that $\mu_2^T\mathcal{D}-\mu_1^T<0$, we get that the following equivalent condition
\begin{equation}
    \lambda^TA_0+\mu_2^T\begin{bmatrix}
         \col_{i=1}^N(\bar{A}_i)\\
    \col_{i=1}^N(\bar{C}_i)
    \end{bmatrix}+\mathds{1}_{n_y}^TC_0<0
\end{equation}
and using the second inequality we can eliminate $\mu_2$ to get the equivalent inequalities
\begin{equation}
  \begin{array}{rcl}
        \textstyle\lambda^T(A_0+\sum_{i=1}^N\bar{A}_i)+\mathds{1}_{n_y}^T(C_0+\sum_{i=1}^N\bar{C}_i) &<&0\\
         \lambda^T E_u+\mathds{1}_{n_y}^TF_u-\gamma_1\mathds{1}_{n_u}^T&<&0.
  \end{array}
\end{equation}
The proof is completed.

\noindent\textbf{Proof of the statement (ii).} From the Theorem \ref{th:ILC} and Remark \ref{remark:dszkdls}, this is equivalent
\begin{equation}
  \begin{bmatrix}
    A_0^TP+PA_0 & \mathds{1}_N^T\otimes P &0 & PE_u & (\mathds{1}_{2N}^T\otimes I_n)Z_1 &C_0^T\\
    \star &  -Z_2^1 &  -Z_2^2 & 0 & 0 &0\\
    \star &  \star &  -Z_2^3 & 0 & 0 & \mathds{1}_N\otimes I_{n_y} \\
    \star &  \star&  \star & -\gamma_2I_{n_u} &  0 & F_u^T\\
    \star &  \star&  \star & \star &  -Z_1 & 0\\
    \star &  \star&  \star & \star &  \star& -\gamma_2I_{n_y}\\
  \end{bmatrix}\prec0
\end{equation}
and $\mathcal{D}^TZ_2\mathcal{D}-Z_1\prec0$ where $\mathcal{D}=\diag\left(\diag_{i=1}^N\left(\bar{A}_i\right),\diag_{i=1}^N\left(\bar{C}_i\right)\right)$. Performing a Schur complement and combining these inequalities yields
\begin{equation}
  \begin{bmatrix}
    A_0^TP+PA_0+(\mathds{1}_{2N}^T\otimes I_n)\mathcal{D}^TZ_2\mathcal{D}(\mathds{1}_{2N}\otimes I_n) & \mathds{1}_N^T\otimes P & 0 & PE_u & C_0^T\\
   \star &  -Z_2^1 &  -Z_2^2 & 0 &0\\
    \star &  \star &  -Z_2^3 & 0 & \mathds{1}_N\otimes I_{n_y}\\
     \star &  \star&  \star & -\gamma_2I_{n_u} & F_u^T\\
     \star &  \star&  \star & \star & -\gamma_2I_{n_y}
  \end{bmatrix}\prec0.
\end{equation}
Since the system is positive, then we can restrict ourselves to a diagonal $Z_2=\diag(\diag_{i=1}^N(Q_i),\diag_{i=1}^N(R_i))$, $Q_iR_i,\in\mathbb{D}^n_{\succ0}$, and hence
\begin{equation}
  \begin{bmatrix}
    A_0^TP+PA_0+\sum_{i=1}^N\bar{A}_i^TQ_i\bar{A}_i+\sum_{i=1}^N\bar{C}_i^TR_i\bar{C}_i  & \mathds{1}_N^T\otimes P & 0& PE_u & C_0^T\\
   \star &  -\diag_{i=1}^N(Q_i) &  0 & 0 &0\\
    \star &  \star &  -\diag_{i=1}^N(R_i) & 0 & \mathds{1}_N\otimes I_{n_y}\\
     \star &  \star&  \star & -\gamma_2I_{n_u} & F_u^T\\
     \star &  \star&  \star & \star & -\gamma_2I_{n_y}
  \end{bmatrix}\prec0.
\end{equation}

\noindent\textbf{Proof of the statement (iii).} It is similar to the proof of the statement (i) and it is thus omitted.
\end{proof}}
%
%
%
%

\blue{\section{Stability and performance of neutral linear positive systems}\label{sec:neutral}

\subsection{Stability analysis}

Neutral systems have been extensively studied as they arise, for instance, transmission lines, models of population dynamics, etc. \cite{Hale:93,Hale:02,Hale:77,Niculescu:01,Bellen:99,Verriest:07}. The special case of linear positive neutral systems has been considered in  \cite{Ebihara:16,Ebihara:17} in the single constant delay case. We extend here these stability analysis results to the case of multiple delays, possibly time-varying, and to performance analysis.

Let us start with the following result:
\begin{proposition}\label{prop:neutral1}
The time-delay system
\begin{equation}\label{eq:TDS_neutral}
  \dot{x}(t)=A_0x(t)+\sum_{i=1}^NA_{r,i}x(t-h_i)+\sum_{i=1}^NA_{n,i}\dot{x}(t-h_i)
\end{equation}
coincides with the uncertain system \eqref{eq:uncertain} where $A=A_0$, $E=[I_n\ \ldots\ I_n]$, $C=\col_{i=1}^N(A_{n,i}A_0+A_{r,i})$ and $F=\col_{i=1}^N(\mathds{1}^T_N\otimes A_{n,i})$ and
  \begin{equation*}
    \Delta\in\boldsymbol{\Delta_d}:=\left\{\diag_{i=1}^N(e^{-sh_i}I_n):\ h\ge0,\ \Re(s)\ge0\right\}.
  \end{equation*}
\end{proposition}

It has been proven in \cite{Ebihara:16,Ebihara:17} that the system \eqref{eq:TDS_neutral} is positive if and only if $A_0$ is Metzler and the matrices $A_{n,i}A_0+A_{r,i}$ and $A_{n,i}$ are nonnegative for all $i=1,\ldots,N$.
%
%
We then have the following result:
\begin{theorem}\label{th:TDSneutral}
 Assume that the system \eqref{eq:TDS_neutral} is positive. Then, the following statements are equivalent:
  \begin{enumerate}[(i)]
    \item The system \eqref{eq:TDS_neutral} is asymptotically stable.
    \item The nonnegative matrix $\textstyle\sum_{i=1}^NA_{n,i}$ is Schur stable and the Metzler matrix
    \begin{equation}\label{eq:dkposdksmkdskdskm1}
    \left(I_n- \sum_{i=1}^NA_{n,i}\right)^{-1}\left(A_0+\sum_{i=1}^NA_{r,i}\right)
    %
    \end{equation}
    is Hurwitz stable.
    %
        %
    %
    \item There exist matrices $P,Q_i\in\mathbb{D}^n_{\succ0}$, $i=1,\ldots,N$, such that the LMI
    \begin{equation}\label{eq:kdsodlsmdkkdsmd_neutral}
    \begin{bmatrix}
      A_0^TP+PA_0 & \star & \star\\
      \col_{i=1}^N(P) & -\diag_{i=1}^N(Q_i) &\star\\
       \col_{i=1}^N(Q_i(A_{n,i}A_0+A_{r,i})) &  \col_{i=1}^N(Q_i(\mathds{1}_N^T\otimes A_{n,i})) & -\diag_{i=1}^N(Q_i)\\
    \end{bmatrix}\prec0
    \end{equation}
    holds.
    \item $A_0$ is Hurwitz stable and the matrices
    \begin{equation}\label{eq:dkposdksmkdskdskm2}
 \sum_{i=1}^NA_{n,i}
  \quad\textnormal{and}\quad-\left(\sum_{i=1}^NA_{r,i}\right)A_0^{-1}
    \end{equation}
    are Schur stable.
  \end{enumerate}
\end{theorem}
\begin{proof}
\textbf{Proof that (i) is equivalent to (ii).} To prove the equivalence between (i) and (ii), first note that (i) is equivalent to saying that there exist vectors $\lambda\in\mathbb{R}^n_{>0}$, $\mu_i\in\mathbb{R}^n_{>0}$, $i=1,\ldots,N$ such that
    \begin{equation}
      \begin{bmatrix}
            A_0 & \vline & \textstyle\row_{i=1}^N(I_n)\\
            \hline
            \textstyle\col_{i=1}^N(A_{n,i}A_0+A_{r,i}) &\vline & \textstyle\col_{i=1}^N(\mathds{1}_N^T\otimes A_{n,i})-I_{Nn}
      \end{bmatrix}\begin{bmatrix}
        \lambda\\
        \hline
\textstyle     \col_{i=1}^N(\mu_i)
      \end{bmatrix}<0
    \end{equation}
or, equivalently, that the above Metzler matrix is Hurwitz stable. Let $\textstyle M_{12} = \row_{i=1}^N(I_n)$, $M_{21}=\col_{i=1}^N(A_{n,i}A_0+A_{r,i})$ and $\textstyle M_{22}=\col_{i=1}^N(\mathds{1}_N^T\otimes A_{n,i})-I_{Nn}$. Then, the stability of the above matrix is equivalent to the stability of the matrix $M_{22}$ together with the stability of the matrix $A_0-M_{12}M_{22}^{-1}M_{21}$. Clearly, the Metzler matrix $M_{22}$ is Hurwitz stable if and only if the nonnegative matrix $\textstyle \col_{i=1}^N(\mathds{1}_N^T\otimes A_{n,i})$ is Schur stable. Since $\textstyle \col_{i=1}^N(A_{n,i})(\mathds{1}_N^T\times I_n)$ and for two matrices $Z_1,Z_2$, we have that $\rho(Z_1Z_2)=\rho(Z_2Z_1)$, then we get that $M_{22}$ is Hurwitz stable if and only if $\textstyle \rho(\sum_{i=1}^NA_{n,i})<1$. This proves the first part. To prove the second part, we need to evaluate $\textstyle A_0-M_{12}M_{22}^{-1}M_{21}$. Using the Sherman-Morrison formula, we get that
\begin{equation}
  M_{22}^{-1}=-\left(I_{Nn}+\col_{i=1}^N(A_{n,i})\left(I_n-\sum_{i=1}^NA_{n,i}\right)^{-1}\row_{i=1}^N(I_n)\right).
\end{equation}
Hence, we have that $A_0-M_{12}M_{22}^{-1}M_{21}$ is equal to
\begin{equation}
  \begin{array}{rcl}
  A_0+\left[I_n+\left(\sum_{i=1}^NA_{n,i}\right)\left(I_n-\sum_{i=1}^NA_{n,i}\right)^{-1}\right]\sum_{i=1}^N(A_{n,i}A_0+A_{r,i})
  \end{array}
\end{equation}
which then simplifies to
\begin{equation}
  \begin{array}{rcl}
  A_0+\left(I_n-\sum_{i=1}^NA_{n,i}\right)^{-1}\sum_{i=1}^N(A_{n,i}A_0+A_{r,i})
  \end{array}
\end{equation}
and to
\begin{equation}
  \begin{array}{rcl}
  \left(I_n-\sum_{i=1}^NA_{n,i}\right)^{-1}\left(A_0+\sum_{i=1}^NA_{r,i}\right).
  \end{array}
\end{equation}
The proof of the equivalence is completed.

\noindent\textbf{Proof that (i) is equivalent to (iii).} This follows from the bounded real lemma.

\noindent\textbf{Proof that (i) is equivalent to (iv).} Using the fact that (i) is equivalent to saying that $\rho(F)<1$ and $\rho(-CA^{-1}E+F)<1$, we get that (i) is equivalent to saying that $A_0$ is Hurwitz stable, that $\textstyle \rho(\sum_{i=1}^NA_{n,i})<1$ and that
  \begin{equation}
        -\begin{bmatrix}
          A_{n,1}A_0+A_{r,1}\\
          \vdots\\
          A_{n,N}A_0+A_{r,N}
        \end{bmatrix}A_0^{-1}\begin{bmatrix}
          I_n & \ldots & I_n
        \end{bmatrix}+\col_{i=1}^N(\mathds{1}_N^T\otimes A_{n,i})
    \end{equation}
    is Schur stable. Expanding it yields
      \begin{equation}
        -\begin{bmatrix}
        A_{r,1}\\
          \vdots\\
        A_{r,N}
        \end{bmatrix}A_0^{-1}\begin{bmatrix}
          I_n & \ldots & I_n
        \end{bmatrix}
    \end{equation}
    and the proof is completed.
\end{proof}

Interestingly, we can see that, once again, the magnitude of the delays does not affect the stability of the process and that the stability of the system can be inferred from the stability of the system with all the delays set to zero. Another interesting point is regarding the concept of strong stability of a difference equation (see Theorem \ref{th:strongly} and  \cite{Hale:93,Hale:02}). A difference equation of the form
\begin{equation}
  x(t)=\sum_{i=1}^NM_ix(t-h_i)
\end{equation}
is said to be strongly stable if and only if
\begin{equation}
\max_{\omega\in[0,2\pi]^N}\rho\left(\sum_{k=1}^NM_ke^{-i\omega_k}\right)<1.
\end{equation}
The notion of strong stability has been introduced in  \cite{Hale:93,Hale:02}  for the analysis of neutral delay systems as the strong stability of the delay-difference equation acting on the derivative of the state is a necessary condition for the stability of overall neutral delay system and the robustness with respect to arbitrarily small changes in the values of the delays. In the present case, we have that
\begin{equation}
\max_{\omega\in[0,2\pi]^N}\rho\left(\sum_{k=1}^NA_{n,k}e^{-i\omega_k}\right)=\rho\left(\sum_{k=1}^NA_{n,k}\right)
\end{equation}
since the matrices  $A_k^n$ are nonnegative and hence the maximum is attained at $\theta_k=0$. Hence, we can see that the condition of strong stability is encoded in the condition in terms of the well-posedness of the interconnection of system \eqref{eq:uncertain} with the matrices and operators described in Proposition \ref{prop:neutral1}.

Finally, it seems interesting to mention that the conditions of the theorem remains true for the stability of neutral systems with time-varying delays provided that $t-h_i(t)\to\infty$ as $t\to\infty$, $i=1,\ldots,N$.

\subsection{Performance analysis}

Let us address now the performance analysis of neutral systems. Let us start with the following result:
\begin{proposition}\label{prop:neutral2}
The time-delay system
\begin{equation}\label{eq:TDS_neutral2}
\begin{array}{rcl}
    \dot{x}(t)&=&A_0x(t)+\sum_{i=1}^NA_{r,i}x(t-h_i)+\sum_{i=1}^NA_{n,i}\dot{x}(t-h_i)+E_uu(t)\\
    z(t)&=&C_0x(t)+\sum_{i=1}^NC_{r,i}x(t-h_i)+\sum_{i=1}^NC_{n,i}\dot{x}(t-h_i)+F_uu(t)\\
\end{array}
\end{equation}
coincides with the uncertain system \eqref{eq:uncertain} where $A=A_0$, $E=\begin{bmatrix}
  \mathds{1}_N^T\otimes I_n & \vline & 0_N^T\otimes I_n & \vline & E_u
\end{bmatrix}$,
\begin{equation}
  C=\begin{bmatrix}
    \col_{i=1}^N(A_{n,i}A_0+A_{r,i})\\
    \hline
    \col_{i=1}^N(C_{n,i}A_0+C_{r,i})\\
    \hline
    C_0
  \end{bmatrix}F=\begin{bmatrix}
  \col_{i=1}^N(\mathds{1}^T_N\otimes A_{n,i}) & \vline & 0 & \vline & 0\\
  \hline
  \col_{i=1}^N(\mathds{1}^T_N\otimes C_{n,i}) & \vline & 0 & \vline & 0\\
  \hline
  0 & \vline & \mathds{1}_N^T\otimes I_n & \vline & F_u
  \end{bmatrix}
\end{equation}
  \begin{equation*}
    \Delta\in\boldsymbol{\Delta_d}:=\left\{e^{-sh_i}I_{N(n+n_z)}:\ h\ge0,\ \Re(s)\ge0\right\}.
  \end{equation*}
\end{proposition}
Moreover, the system is positive if and only if $A_0$ is Metzler and the matrices $E,C$ and $F$ are nonnegative. We then have the following result:
\begin{theorem}
  Assume that the system \eqref{eq:TDS_neutral} is positive. Then, the following statements are equivalent:
  \begin{enumerate}[(i)]
  \item\label{item:neutral:L1}  The system \eqref{eq:TDS_neutral2} is asymptotically stable and has $L_1$-gain smaller than $\gamma_1$ if and only if
      \begin{itemize}
        \item the matrix $\sum_{i=1}^NA_{n,i}$ is Schur stable and
        \item there exists a vector $\lambda\in\mathbb{R}_{>0}^n$ such that
        \begin{equation}
          \begin{bmatrix}
            \lambda\\
            \mathds{1}_{n_y}
          \end{bmatrix}^T\begin{bmatrix}
 S^{-1}\left(A_0+\sum_{i=1}^NA_{r,i}\right) &\vline&  S^{-1}E_u\\
 \hline
 C_0+\sum_{i=1}^NC_{r,i}+\left(\sum_{i=1}^NC_{n,i}\right)S^{-1}\left(\sum_{i=1}^NA_{r,i}\right) & \vline & F_u+\left(\sum_{i=1}^NC_{n,i}\right)S^{-1}E_u
          \end{bmatrix}<\begin{bmatrix}
            0\\
            \gamma_1\mathds{1}_{n_u}
            \end{bmatrix}^T
        \end{equation}
        where $S:=I-\sum_{i=1}^NA_{n,i}$.
      \end{itemize}
  \item\label{item:neutral:L2}   The system \eqref{eq:TDS_neutral2} is asymptotically stable and has $L_2$-gain smaller than $\gamma_2$ if and only if there exists diagonal matrices $P\in\mathbb{D}^n_{\succ0}$, $Q\in\mathbb{D}^{nN}_{\succ0}$ and $R\in\mathbb{D}^{Nn_y}_{\succ0}$ such that the LMI
      \begin{equation}
      \begin{bmatrix}
        A_0^TP+PA_0 & P(\mathds{1}_N^T\otimes I_n) & 0 & PE_u\\
        \star & -Q & 0 & 0\\
        \star & \star & -R & 0\\
        \star & \star & \star & -\gamma_2^2 I_{n_u}
      \end{bmatrix}+\begin{bmatrix}
        C^T\\
        F^T
      \end{bmatrix}\diag(Q,R,I_{n_y})\begin{bmatrix}
        C\\
        F
      \end{bmatrix}\prec0
      \end{equation}
      where the matrices $A,E,C,F$ are defined in Proposition \ref{prop:neutral2}.
    \item\label{item:neutral:Linf} The system \eqref{eq:TDS_neutral2} is asymptotically stable and has $L_\infty$-gain smaller than $\gamma_\infty$ if and only if
        \begin{itemize}
          \item the matrix $\sum_{i=1}^NA_{n,i}$ is Schur stable and
          \item there exists a vector $\lambda\in\mathbb{R}_{>0}^n$ such that
        \begin{equation}
          \begin{bmatrix}
     S^{-1}\left(A_0+\sum_{i=1}^NA_{r,i}\right) &\vline&  S^{-1}E_u\\
         \hline
        C_0+\sum_{i=1}^NC_{r,i}+\left(\sum_{i=1}^NC_{n,i}\right)S^{-1}\left(\sum_{i=1}^NA_{r,i}\right) & \vline & F_u+\left(\sum_{i=1}^NC_{n,i}\right)S^{-1}E_u
          \end{bmatrix}\begin{bmatrix}
            \lambda\\
            \mathds{1}_{n_u}
          \end{bmatrix}<\begin{bmatrix}
            0\\
            \gamma_\infty\mathds{1}_{n_y}
            \end{bmatrix}
        \end{equation}
        where $S:=I-\sum_{i=1}^NA_{n,i}$.
        \end{itemize}
\end{enumerate}
\end{theorem}
\begin{proof}
The statement (ii) can be obtained using the scaled-small gain in the $L_2$-framework. We now prove the statement \eqref{item:neutral:Linf}, statement \eqref{item:neutral:L1} can be proven in exactly the same way. First note that the system \eqref{eq:TDS_neutral2} is asymptotically stable and has $L_\infty$-gain smaller than $\gamma_\infty$ if and only if there exist some vectors $\lambda\in\mathbb{R}_{>0}^n$, $\mu_i^1\in\mathbb{R}_{>0}^n$ and $\mu_i^2\in\mathbb{R}_{>0}^{n_u}$, $i=1,\ldots,N$ such that the following inequality
 \begin{equation}
      \begin{bmatrix}
            A_0 & E_u & \vline & \textstyle\row_{i=1}^N(I_n) & 0\\
            C_0 &F_u & \vline & 0 & \textstyle\row_{i=1}^N(I_{n_y})\\
            \hline
            \textstyle\col_{i=1}^N(A_{n,i}A_0+A_{r,i}) & \textstyle\col_{i=1}^N(A_{n,i}E_u) & \vline & \textstyle\col_{i=1}^N(\mathds{1}_N^T\otimes A_{n,i})-I_{Nn} & 0\\
            \textstyle\col_{i=1}^N(C_{n,i}A_0+C_{r,i})  & \textstyle\col_{i=1}^N(C_{n,i}E_u) & \vline & \textstyle\col_{i=1}^N(\mathds{1}_N^T\otimes C_{n,i}) & -I_{n_yN}
      \end{bmatrix}\begin{bmatrix}
              \lambda\\
              \mathds{1}_{n_u}\\
              \hline
        \textstyle     \col_{i=1}^N(\mu_i^1)\\
        \textstyle     \col_{i=1}^N(\mu_i^2)
      \end{bmatrix}<\begin{bmatrix}
            0\\
            \gamma\mathds{1}_{n_y}\\
            \hline
        0\\
        0\\
         \end{bmatrix}
\end{equation}
is satisfied. Let us denote for simplicity the above matrix by $[\mathcal{M}_{ij}]_{i,j=1,2}$. Solving for the $\mu$ terms yields that the above condition is equivalent to saying that
\begin{equation}
 \left(\mathcal{M}_{11}-\mathcal{M}_{12}\mathcal{M}_{22}^{-1}\mathcal{M}_{21}\right)\begin{bmatrix}
              \lambda\\
              \mathds{1}_{n_u}
              \end{bmatrix}<\begin{bmatrix}
            0\\
            \gamma\mathds{1}_{n_y}\end{bmatrix}
\end{equation}
together with $\mathcal{M}_{22}$ is Hurwitz stable. It is immediate to see that, once again, the latter condition is equivalent to the Schur stability of the matrix $\textstyle\sum_{i=1}^NA_{n,i}$. For compactness, let us denote now $\textstyle\mathcal{A}:= \col_{i=1}^N(\mathds{1}_N^T\otimes A_{n,i})$ and $\textstyle\mathcal{C}:= \col_{i=1}^N(\mathds{1}_N^T\otimes C_{n,i})$ and, using this notation, we get that
\begin{equation}
  \mathcal{M}_{22}=\begin{bmatrix}
  \col_{i=1}^N(\mathds{1}_N^T\otimes A_{n,i})-I_{nN} & \quad 0\\
    \col_{i=1}^N(\mathds{1}_N^T\otimes C_{n,i})& \quad-I_{nN}
\end{bmatrix}
\end{equation}
and, hence, we have that
\begin{equation}
  \mathcal{M}_{22}^{-1}=\begin{bmatrix}
(\mathcal{A}-I_{nN})^{-1} & 0\\
\mathcal{C}(\mathcal{A}-I_{nN})^{-1} & -I_{Nn}
\end{bmatrix}
\end{equation}
together with
\begin{equation}
(\mathcal{A}-I_{nN})^{-1}=  I_{Nn}+\col_{i=1}^N(A_{n,i})\left(I_n-\sum_{i=1}^NA_{n,i}\right)^{-1}\row_{i=1}^N(I_n).
\end{equation}
Letting now $\textstyle S:=I-\sum_{i=1}^NA_{n,i}$, we then get that
\begin{equation}
  S^{-1}\left[\left(A_0+\sum_{i=1}^NA_{r,i}\right)\lambda+E_u\mathds{1}_{n_u}\right]<0
\end{equation}
and similar manipulation gives
\begin{equation}
\left[C_0+\sum_{i=1}^NC_{r,i}+\left(\sum_{i=1}^NC_{n,i}\right)S^{-1}\left(\sum_{i=1}^NA_{r,i}\right)\right]\lambda+\left(F_u+\left(\sum_{i=1}^NC_{n,i}\right)S^{-1}E_u\right)\mathds{1}_{n_u}<\gamma\mathds{1}_{n_y}.
\end{equation}
The proof is now completed.
\end{proof}

As for the other systems, the $L_\infty$ condition remains the same in the case of time-varying delays provided that  $t-h_i(t)\to\infty$ as $t\to\infty$, $i=1,\ldots,N$. On the other hand, the other conditions need to be slightly changed to incorporate the rate of variation of the delays as in the other results.}

\section{Conclusion}

Several recent results regarding the robust stability analysis of uncertain linear positive systems have been unified in a single formulation using a generalization of the structured singular value. Using this generalization, several necessary and sufficient conditions have been obtained and expressed in terms of scaled small-gain theorems involving linear or semidefinite programs. These results have been considered for establishing several results for linear positive systems with constant and time-varying delays. \blue{ It is notably recalled that the time-varying nature of the delay never deteriorates the asymptotic stability of linear positive systems but may deteriorate their $L_p$ stability.}

Interesting extensions could be concerned with the robust stabilization problem using static/dynamic output-feedback or state-feedback controllers using ideas from \cite{Aitrami:07,Briat:11h,Ebihara:12,Naghnaeian:14} or the extension of the results to hybrid systems \cite{Briat:17NAHS}. The design of interval observers is also a potentially interesting follow-up to this work; see e.g. \cite{Mazenc:11,Briat:15g,Efimov:15,Efimov:16b,Briat:17ifacObs}.

\bibliographystyle{unsrtnat}

\begin{thebibliography}{90}
\providecommand{\natexlab}[1]{#1}
\providecommand{\url}[1]{\texttt{#1}}
\expandafter\ifx\csname urlstyle\endcsname\relax
  \providecommand{\doi}[1]{doi: #1}\else
  \providecommand{\doi}{doi: \begingroup \urlstyle{rm}\Url}\fi

\bibitem[Farina and Rinaldi(2000)]{Farina:00}
L.~Farina and S.~Rinaldi.
\newblock \emph{Positive Linear Systems: Theory and Applications}.
\newblock John Wiley \& Sons, 2000.

\bibitem[Murray(2002)]{Murray:02}
J.~D. Murray.
\newblock \emph{Mathematical Biology Part I. An Introduction. 3rd Edition}.
\newblock Springer-Verlag Berlin Heidelberg, 2002.

\bibitem[Briat and Verriest(2009)]{Briat:09h}
C.~Briat and E.~I. Verriest.
\newblock A new delay-{SIR} model for pulse vaccination.
\newblock \emph{Biomedical signal processing and control}, 4(4):\penalty0
  272--277, 2009.

\bibitem[Briat and Khammash(2012)]{Briat:12c}
C.~Briat and M.~Khammash.
\newblock Computer control of gene expression: Robust setpoint tracking of
  protein mean and variance using integral feedback.
\newblock In \emph{51st {IEEE} Conference on Decision and Control}, pages
  3582--3588, Maui, Hawaii, USA, 2012.

\bibitem[Gupta et~al.(2014)Gupta, Briat, and Khammash]{Briat:13i}
A.~Gupta, C.~Briat, and M.~Khammash.
\newblock A scalable computational framework for establishing long-term
  behavior of stochastic reaction networks.
\newblock \emph{PLOS Computational Biology}, 10(6):\penalty0 e1003669, 2014.

\bibitem[Briat et~al.(2016{\natexlab{a}})Briat, Gupta, and Khammash]{Briat:15e}
C.~Briat, A.~Gupta, and M.~Khammash.
\newblock Antithetic integral feedback ensures robust perfect adaptation in
  noisy biomolecular networks.
\newblock \emph{Cell Systems}, 2:\penalty0 17--28, 2016{\natexlab{a}}.

\bibitem[Briat et~al.(2016{\natexlab{b}})Briat, Zechner, and
  Khammash]{Briat:16a}
C.~Briat, C.~Zechner, and M.~Khammash.
\newblock Design of a synthetic integral feedback circuit: dynamic analysis and
  {DNA} implementation.
\newblock \emph{ACS Synthetic Biology}, 5(10):\penalty0 1108--1116,
  2016{\natexlab{b}}.

\bibitem[Gouz{\'{e}} et~al.(2000)Gouz{\'{e}}, Rapaport, and
  {Hadj-Sadok}]{Gouze:00}
J.~L. Gouz{\'{e}}, A.~Rapaport, and M.~Z. {Hadj-Sadok}.
\newblock Interval observers for uncertain biological systems.
\newblock \emph{Ecological modelling}, 133:\penalty0 45--56, 2000.

\bibitem[Mazenc and Bernard(2011)]{Mazenc:11}
F.~Mazenc and O.~Bernard.
\newblock Interval observers for linear time-invariant systems with
  disturbances.
\newblock \emph{Automatica}, 47:\penalty0 140--147, 2011.

\bibitem[Briat and Khammash(2016)]{Briat:15g}
C.~Briat and M.~Khammash.
\newblock Interval peak-to-peak observers for continuous- and discrete-time
  systems with persistent inputs and delays.
\newblock \emph{Automatica}, 74:\penalty0 206--213, 2016.

\bibitem[Efimov et~al.(2015)Efimov, Polyakov, and Richard]{Efimov:15}
D.~Efimov, A.~Polyakov, and J.-P. Richard.
\newblock Interval observer design for estimation and control of time-delay
  descriptor systems.
\newblock \emph{European Journal of Control}, 23:\penalty0 26--35, 2015.

\bibitem[Efimov et~al.(2016{\natexlab{a}})Efimov, Fridman, Polyakov,
  Perruquetti, and Richard]{Efimov:16}
D.~Efimov, E.~Fridman, A.~Polyakov, W.~Perruquetti, and J.-P. Richard.
\newblock On design of interval observers with sampled measurement.
\newblock \emph{Systems \& Control Letters}, 96:\penalty0 158--164,
  2016{\natexlab{a}}.

\bibitem[Efimov et~al.(2016{\natexlab{b}})Efimov, Fridman, Polyakov,
  Perruquetti, and Richard]{Efimov:16b}
D.~Efimov, E.~Fridman, A.~Polyakov, W.~Perruquetti, and J.-P. Richard.
\newblock Linear interval observers under delayed measurements and
  delay-dependent positivity.
\newblock \emph{Automatica}, 72:\penalty0 123--130, 2016{\natexlab{b}}.

\bibitem[Ngoc and Trinh(2016)]{Ngoc:16}
P.~H.~A. Ngoc and H.~Trinh.
\newblock Novel criteria for exponential stability of linear neutral
  time-varying differential systems.
\newblock \emph{IEEE Transactions on Automatic Control}, 61(6):\penalty0
  1590--1594, 2016.

\bibitem[Mazenc and Malisoff(2016)]{Mazenc:16}
F.~Mazenc and M.~Malisoff.
\newblock Stability analysis for time-varying systems with delay using linear
  {L}yapunov functionals and a positive systems approach.
\newblock \emph{IEEE Transactions on Automatic Control}, 61(3):\penalty0
  771--776, 2016.

\bibitem[{Ait Rami} and Tadeo(2007)]{Aitrami:07}
M.~{Ait Rami} and F.~Tadeo.
\newblock Controller synthesis for positive linear systems with bounded
  controls.
\newblock \emph{{IEEE} Transactions on Circuits and Systems -- II. Express
  Briefs}, 54(2):\penalty0 151--155, 2007.

\bibitem[{Ait Rami}(2011)]{AitRami:11}
M.~{Ait Rami}.
\newblock Solvability of static output-feedback stabilization for {LTI}
  positive systems.
\newblock \emph{Systems \& Control Letters}, 60:\penalty0 704--708, 2011.

\bibitem[Briat(2013)]{Briat:11h}
C.~Briat.
\newblock Robust stability and stabilization of uncertain linear positive
  systems via integral linear constraints - ${L_1}$- and ${L_\infty}$-gains
  characterizations.
\newblock \emph{{I}nternational {J}ournal of {R}obust and {N}onlinear
  {C}ontrol}, 23(17):\penalty0 1932--1954, 2013.

\bibitem[Briat(2011{\natexlab{a}})]{Briat:11g}
C.~Briat.
\newblock Robust stability analysis of uncertain linear positive systems via
  integral linear constraints - ${L_1}$- and ${L_\infty}$-gains
  characterizations.
\newblock In \emph{50th {IEEE} Conference on Decision and Control}, pages
  3122--3129, Orlando, Florida, USA, 2011{\natexlab{a}}.

\bibitem[Ebihara et~al.(2011)Ebihara, Peaucelle, and Arzelier]{Ebihara:11}
Y.~Ebihara, D.~Peaucelle, and D.~Arzelier.
\newblock ${L_1}$ gain analysis of linear positive systems and its
  applications.
\newblock In \emph{50th Conference on Decision and Control, Orlando, Florida,
  USA}, pages 4029--4034, 2011.

\bibitem[Rantzer(2016)]{Rantzer:16}
A.~Rantzer.
\newblock On the {K}alman-{Y}akubovich-{P}opov lemma for positive systems.
\newblock \emph{IEEE Transactions on Automatic Control}, 61(5):\penalty0
  1346--1349, 2016.

\bibitem[Tanaka and Langbort(2010)]{Tanaka:10}
T.~Tanaka and C.~Langbort.
\newblock {KYP} {L}emma for internally positive systems and a tractable class
  of distributed {H}-infinity control problems.
\newblock In \emph{American Control Conference}, pages 6238--6243, Baltimore,
  Maryland, USA, 2010.

\bibitem[Colombino and Smith(2016)]{Colombino:15}
M.~Colombino and R.~S. Smith.
\newblock A convex characterization of robust stability for positive and
  positively dominated linear systems.
\newblock \emph{IEEE Transactions on Automatic Control}, 61(7):\penalty0
  1965--1971, 2016.

\bibitem[Khong et~al.(2015)Khong, Briat, and Rantzer]{Briat:15cdc}
S.~Z. Khong, C.~Briat, and A.~Rantzer.
\newblock Positive systems analysis via integral linear constraints.
\newblock In \emph{54th {IEEE} Conference on Decision and Control}, pages
  6373--6378, Osaka, Japan, 2015.

\bibitem[Megretski and Treil(1993)]{Megretski:93}
A.~Megretski and S.~Treil.
\newblock Power distribution in optimization and robustness of uncertain
  systems.
\newblock \emph{Journal of Mathematical Systems, Estimation and Control},
  3:\penalty0 301--319, 1993.

\bibitem[Megretski and Rantzer(1997)]{RantzerMegretski:97}
A.~Megretski and A.~Rantzer.
\newblock System analysis via {I}ntegral {Q}uadratic {C}onstraints.
\newblock \emph{IEEE Transactions on Automatic Control}, 42(6):\penalty0
  819--830, 1997.

\bibitem[Packard and Doyle(1993)]{Packard:93}
A.~Packard and J.~C. Doyle.
\newblock The complex structured singular value.
\newblock \emph{Automatica}, 29:\penalty0 71--109, 1993.

\bibitem[Dahleh and Bobillo(1995)]{Dahleh:95}
M.~A. Dahleh and I.~J.~Diaz Bobillo.
\newblock \emph{Control of uncertain systems - A linear programming approach}.
\newblock Prentice-Hall, 1995.

\bibitem[Khammash(1993)]{Khammash:93}
M.~H. Khammash.
\newblock Necessary and sufficient conditions for the robustness of
  time-varying systems with applications to sampled-data systems.
\newblock \emph{{IEEE} Transactions on Automatic Control}, 38(1):\penalty0
  49--57, 1993.

\bibitem[Haddad and Chellaboina(2004)]{Haddad:04}
W.~M. Haddad and V.~Chellaboina.
\newblock Stability theory for nonnegative and compartmental dynamical systems
  with time delay.
\newblock \emph{Systems \& Control Letters}, 51(5):\penalty0 355--361, 2004.

\bibitem[{Ait Rami}(2009)]{AitRami:09}
M.~{Ait Rami}.
\newblock Stability analysis and synthesis for linear positive systems with
  time-varying delays.
\newblock In \emph{Positive systems - Proceedings of the 3rd
  {M}ultidisciplinary {I}nternational {S}ymposium on {P}ositive {S}ystems:
  {T}heory and {A}pplications ({POSTA} 2009)}, pages 205--216. Springer-Verlag
  Berlin Heidelberg, 2009.

\bibitem[Mason(2012)]{Mason:12}
O.~Mason.
\newblock Diagonal {R}iccati stability and positive time-delay systems.
\newblock \emph{Systems \& Control Letters}, 61:\penalty0 6--10, 2012.

\bibitem[Zhu and Chen(2015)]{Zhu:15}
J.~Zhu and J.~Chen.
\newblock Stability of systems with time-varying delays: An {$\mathscr{L}_1$}
  small-gain perspective.
\newblock \emph{Automatica}, 52:\penalty0 260--265, 2015.

\bibitem[Shen and Lam(2015)]{Shen:15}
J.~Shen and J.~Lam.
\newblock $\ell_\infty$/${L}_\infty$-gain analysis for positive linear systems
  with unbounded time-varying delays.
\newblock \emph{IEEE Transactions on Automatic Control}, 60(3):\penalty0
  857--862, 2015.

\bibitem[Zhang et~al.(2001)Zhang, Knospe, and Tsiotras]{Zhang:01a}
J.~Zhang, C.~R. Knospe, and P.~Tsiotras.
\newblock Stability of time-delay systems: Equivalence between {L}yapunov and
  scaled small-gain conditions.
\newblock \emph{IEEE Transactions on Automatic Control}, 46:\penalty0 482--486,
  2001.

\bibitem[Niculescu(2001)]{Niculescu:01}
S.~I. Niculescu.
\newblock \emph{Delay effects on stability. A robust control approach}, volume
  269.
\newblock Springer-Verlag: Heidelbeg, 2001.

\bibitem[Knospe and Roozbehani(2003)]{Knospe:03}
C.~R. Knospe and M.~Roozbehani.
\newblock Stability of linear systems with interval time-delay.
\newblock In \emph{{IEEE} American Control Conference}, pages 1458--1463,
  Denver, Colorado, 2003.

\bibitem[Gu et~al.(2003)Gu, Kharitonov, and Chen]{GuKC:03}
K.~Gu, V.~L. Kharitonov, and J.~Chen.
\newblock \emph{Stability of Time-Delay Systems}.
\newblock Birkh{\"a}user, Boston, 2003.

\bibitem[Gouaisbaut and Peaucelle(2006{\natexlab{a}})]{GouaiPeau:06}
F.~Gouaisbaut and D.~Peaucelle.
\newblock Stability of time-delay systems with non-small delay.
\newblock In \emph{Conference on Decision and Control, San Diego, California,
  USA}, pages 840--845, 2006{\natexlab{a}}.

\bibitem[Knospe and Roozbehani(2006)]{KnospeR:06}
C.~R. Knospe and M.~Roozbehani.
\newblock Stability of linear systems with interval time delays excluding zero.
\newblock \emph{IEEE Transactions on Automatic Control}, 51:\penalty0
  1271--1288, 2006.

\bibitem[Gouaisbaut and Peaucelle(2006{\natexlab{b}})]{Gouaisbaut:06}
F.~Gouaisbaut and D.~Peaucelle.
\newblock Delay dependent robust stability of time delay-systems.
\newblock In \emph{$5^{th}$ {IFAC} Symposium on Robust Control Design},
  Toulouse, France, 2006{\natexlab{b}}.

\bibitem[Kao and Rantzer(2007)]{Kao:07}
C.~Y. Kao and A.~Rantzer.
\newblock Stability analysis of systems with uncertain time-varying delays.
\newblock \emph{Automatica}, 43:\penalty0 959--970, 2007.

\bibitem[Ariba and Gouaisbaut(2009)]{Ariba:09}
Y.~Ariba and F.~Gouaisbaut.
\newblock Input-output framework for robust stability of time-varying delay
  systems.
\newblock In \emph{48th Conference on Decision and Control}, pages 274--279,
  Shanghai, China, 2009.

\bibitem[Ariba et~al.(2010)Ariba, Gouaisbaut, and Johansson]{Ariba:10}
Y.~Ariba, F.~Gouaisbaut, and K.~H. Johansson.
\newblock Stability interval for time-varying delay systems.
\newblock In \emph{49th Conference on Decision and Control}, pages 1017--1022,
  Atlanta, Georgia, USA, 2010.

\bibitem[Gouaisbaut and Ariba(2011)]{Gouaisbaut:11}
F.~Gouaisbaut and Y.~Ariba.
\newblock Delay range stability of a class of distributed time delay systems.
\newblock \emph{Systems \& Control Letters}, 60:\penalty0 211--217, 2011.

\bibitem[Briat(2015)]{Briat:book1}
C.~Briat.
\newblock \emph{Linear Parameter-Varying and Time-Delay Systems -- Analysis,
  Observation, Filtering \& Control}, volume~3 of \emph{Advances on Delays and
  Dynamics}.
\newblock Springer-Verlag, Heidelberg, Germany, 2015.

\bibitem[Fridman(2014)]{Fridman:14}
E.~Fridman.
\newblock \emph{Introduction to Time-Delay Systems}.
\newblock Birkh{\"a}user, Springer International Publishing Switzerland, 2014.

\bibitem[Zhu et~al.(2015)Zhu, Qi, and Chen]{Zhu:15b}
J.~Zhu, T.~Qi, and J.~Chen.
\newblock Small-gain stability conditions for linear systems with time-varying
  delays.
\newblock \emph{Systems \& Control Letters}, 81:\penalty0 42--48, 2015.

\bibitem[Li et~al.(2016)Li, Gao, and Gu]{Li:16b}
X.~Li, H.~Gao, and K.~Gu.
\newblock Delay-independent stability analysis of linear time-delay systems
  based on frequency discretization.
\newblock \emph{Automatica}, 70:\penalty0 288--294, 2016.

\bibitem[Shen and Lam(2014)]{Shen:14}
J.~Shen and J.~Lam.
\newblock ${L_\infty}$-gain analysis for positive systems with distributed
  delays.
\newblock \emph{Automatica}, 50:\penalty0 175--179, 2014.

\bibitem[Aleksandrov and Mason(2016)]{Aleksandrov:16}
A.~Aleksandrov and O.~Mason.
\newblock Diagonal {R}iccati stability and applications.
\newblock \emph{Linear Algebra and its Applications}, 492:\penalty0 38--51,
  2016.

\bibitem[Shen and Zheng(2015)]{Shen:15b}
J.~Shen and W.~X. Zheng.
\newblock Positivity and stability of coupled differential–difference
  equations with time-varying delays.
\newblock \emph{Automatica}, 57:\penalty0 123--127, 2015.

\bibitem[Ebihara et~al.(2016)Ebihara, Nishio, and Hagiwara]{Ebihara:16}
Y.~Ebihara, N.~Nishio, and T.~Hagiwara.
\newblock Stability analysis of neutral type time-delay positive systems.
\newblock In \emph{5th International Symposium on Positive Systems}, 2016.
\newblock URL \url{http://www.posta2016.org/slides/4.EbiharaR.pdf}.

\bibitem[Ebihara et~al.(2017)Ebihara, Nishio, and Hagiwara]{Ebihara:17}
Y.~Ebihara, N.~Nishio, and T.~Hagiwara.
\newblock Stability analysis of neutral type time-delay positive systems.
\newblock In F.~Cacace, L.~Farina, R.~Setola, and A.~Germani, editors,
  \emph{Positive Systems -- Theory and Applications (POSTA 2016)}. Springer
  International Publishing, 2017.

\bibitem[Desoer and Vidyasagar(1975)]{Desoer:75a}
C.~A. Desoer and M.~Vidyasagar.
\newblock \emph{Feedback Systems : Input-Output Properties}.
\newblock Academic Press, New York, 1975.

\bibitem[Stoer and Witzgall(1962)]{Stoer:62}
J.~Stoer and C.~Witzgall.
\newblock Transformations by diagonal matrices in a normed space.
\newblock \emph{Numerische Mathematik}, 4:\penalty0 158--171, 1962.

\bibitem[Rantzer(2012)]{Rantzer:12}
A.~Rantzer.
\newblock Optimizing positively dominated systems.
\newblock In \emph{51st {IEEE} Conference on Decision and Control}, pages
  271--277, Maui, Hawaii, USA, 2012.

\bibitem[Sootla and Mauroy(2015)]{Sootla:15}
A.~Sootla and A.~Mauroy.
\newblock Properties of eventually positive linear input-output systems.
\newblock \emph{ArXiv e-prints arXiv:1509.08392}, 2015.

\bibitem[Altafini(2016)]{Altafini:16}
C.~Altafini.
\newblock Minimal eventually positive realizations of externally positive.
\newblock \emph{Automatica}, 68:\penalty0 140--148, 2016.

\bibitem[Dullerud and Paganini(2000)]{Dullerud:00}
G.~E. Dullerud and F.~Paganini.
\newblock \emph{A course in robust control theory. A convex approach}.
\newblock Springer, New York, USA, 2000.

\bibitem[Shorten et~al.(2009)Shorten, Mason, and King]{Shorten:09}
R.~Shorten, O.~Mason, and C.~King.
\newblock An alternative proof of the {B}arker, {B}erman, {P}lemmons result on
  diagonal stability and extensions.
\newblock \emph{Linear Algebra and Its Applications}, 430:\penalty0 34--40,
  2009.

\bibitem[Naghnaeian and Voulgaris(2014)]{Naghnaeian:14}
M.~Naghnaeian and P.~G. Voulgaris.
\newblock Performance optimization over positive $l_\infty$ cones.
\newblock In \emph{American Control Conference}, pages 5645--5650, Portland,
  USA, 2014.

\bibitem[Safonov and Athans(1978)]{Safonov:78}
M.~G. Safonov and M.~Athans.
\newblock On stability theory.
\newblock In \emph{17th Conference on Decision and Control}, pages 301--314,
  San Diego, California, USA, 1978.

\bibitem[Iwasaki and Hara(1998)]{Iwasaki:98a}
T.~Iwasaki and S.~Hara.
\newblock Well-posedness of feedback systems: insight into exact robustness
  analysis and approximate computations.
\newblock \emph{IEEE Transactions on Automatic Control}, 43:\penalty0 619--630,
  1998.

\bibitem[H{\"{o}}rmander(1985)]{Hormander:85}
L.~H{\"{o}}rmander.
\newblock \emph{The analysis of linear partial differential operators I.}
\newblock Springer Verlag, 1985.

\bibitem[Blondel and Megretski(2004)]{Blondel:04bk}
V.~Blondel and A.~Megretski.
\newblock \emph{Unsolved Problems in Mathematical Systems and Control Theory}.
\newblock Princeton University Press, 2004.

\bibitem[Aleksandrov et~al.(2016)Aleksandrov, Mason, and
  Vorob'eva]{Aleksandrov:16b}
A.~Aleksandrov, O.~Mason, and A.~Vorob'eva.
\newblock Diagonal {R}iccati stability and the {H}adamard product.
\newblock 2016.
\newblock URL \url{https://arxiv.org/abs/1612.06587v1}.

\bibitem[Verriest(1999)]{Verriest:99b}
E.~I. Verriest.
\newblock Robust stability, adjoints, and {LQ} control of scale-delay systems.
\newblock In \emph{38th IEEE Conference on decision and control}, pages
  209--214, Phoenix, Arizona, USA, 1999.

\bibitem[Briat(2011{\natexlab{b}})]{Briat:11j}
C.~Briat.
\newblock Robust stability analysis in the $*$-norm and {L}yapunov-{R}azumikhin
  functions for the stability analysis and control of time-delay systems.
\newblock In \emph{50th {IEEE} Conference on Decision and Control}, pages
  6319--6324, Orlando, Florida, USA, 2011{\natexlab{b}}.

\bibitem[Shustin and Fridman(2007)]{Shustin:07}
E.~Shustin and E.~Fridman.
\newblock On delay-derivative-dependent stability of systems with fast-varying
  delays.
\newblock \emph{Automatica}, 43:\penalty0 1649--1655, 2007.

\bibitem[Avellar and Hale(1980)]{Avellar:80}
C.~E. Avellar and J.~K. Hale.
\newblock On the zeros of exponential polynomials.
\newblock \emph{Journal of Mathematical Analysis and Applications},
  73:\penalty0 434--452, 1980.

\bibitem[Damak et~al.(2014)Damak, {Di Loreto}, Lombardi, and Andrieu]{Damak:14}
S.~Damak, M.~{Di Loreto}, W.~Lombardi, and V.~Andrieu.
\newblock Exponential ${L}_2$-stability for a class of linear systems governed
  by continuous-time difference equations.
\newblock \emph{Automatica}, 50:\penalty0 3299--3303, 2014.

\bibitem[Damak et~al.(2015)Damak, {Di Loreto}, Lombardi, and Andrieu]{Damak:15}
S.~Damak, M.~{Di Loreto}, W.~Lombardi, and V.~Andrieu.
\newblock Stability of linear continuous-time difference equations with
  distributed delay: Constructive exponential estimates.
\newblock \emph{International Journal of Robust and Nonlinear Control},
  25(17):\penalty0 3195--3209, 2015.

\bibitem[{Melchor-Aguilar}(2016)]{Melchor:16}
D.~{Melchor-Aguilar}.
\newblock A note on stability of functional difference equations.
\newblock \emph{Systems \& Control Letters}, 67:\penalty0 211--215, 2016.

\bibitem[{Melchor-Aguilar}(2013)]{Melchor:13}
D.~{Melchor-Aguilar}.
\newblock Exponential stability of linear continuous time difference systems
  with multiple delays.
\newblock \emph{Systems \& Control Letters}, 62:\penalty0 811--818, 2013.

\bibitem[Hale and Amores(1977)]{Hale:77}
J.~K. Hale and P.~M. Amores.
\newblock Stability in neutral equations.
\newblock \emph{Nonlinear Analysis: Theory, Methods \& Applications},
  1(1):\penalty0 161--172, 1977.

\bibitem[Niculescu and Rasvan(2000{\natexlab{a}})]{Niculescu:00a}
S.-I. Niculescu and V.~Rasvan.
\newblock Delay-independent stability in lossless propagation models with
  applications (i): A complex domain approach.
\newblock In \emph{International Symposium on Mathematical Theory of Networks
  and Systems}, 2000{\natexlab{a}}.

\bibitem[Niculescu and Rasvan(2000{\natexlab{b}})]{Niculescu:00b}
S.-I. Niculescu and V.~Rasvan.
\newblock Delay-independent stability in lossless propagation models with
  applications (ii): A lyapunov-based approach.
\newblock In \emph{International Symposium on Mathematical Theory of Networks
  and Systems}, 2000{\natexlab{b}}.

\bibitem[Pepe and Verriest(2003)]{Pepe:03}
P.~Pepe and E.~I. Verriest.
\newblock On the stability of coupled delay differential and continuous time
  difference equations.
\newblock \emph{IEEE Transactions on Automatic Control}, 48(8):\penalty0
  1422--1427, 2003.

\bibitem[Niculescu et~al.(2006)Niculescu, Fu, and Chen]{Niculescu:06}
S.-I. Niculescu, P.~Fu, and J.~Chen.
\newblock On the stability of linear delay-differential algebraic systems:
  Exact conditions via matrix pencil solutions.
\newblock In \emph{45th IEEE Conference on Decision \& Control}, pages
  834--839, 2006.

\bibitem[Verriest and Pepe(2009)]{Verriest:09d}
E.~I. Verriest and P.~Pepe.
\newblock Time optimal and optimal impulsive control for coupled differential
  difference point delay systems with an application in forestry.
\newblock In \emph{Topics in Time Delay Systems}, pages 255--265. Springer
  Berlin Heidelberg, 2009.

\bibitem[Briat(2017{\natexlab{a}})]{Briat:14c}
C.~Briat.
\newblock Sign properties of {M}etzler matrices with applications.
\newblock \emph{Linear Algebra and its Applications}, 515:\penalty0 53--86,
  2017{\natexlab{a}}.

\bibitem[Solomon and Fridman(2013)]{Solomon:13}
O.~Solomon and E.~Fridman.
\newblock New stability conditions for systems with distributed delays.
\newblock \emph{Automatica}, 49(11):\penalty0 3467--3475, 2013.

\bibitem[Hale and {Verduyn~Lunel}(1991)]{Hale:93}
J.~K. Hale and S.~M. {Verduyn~Lunel}.
\newblock \emph{Introduction to Functional Differential Equations}.
\newblock Springer-Verlag, New York, USA, 1991.

\bibitem[Hale and {Verdyun Lunel}(2002)]{Hale:02}
J.~K. Hale and S.~M. {Verdyun Lunel}.
\newblock Strong stabilization of neutral functional differential equations.
\newblock \emph{{IMA} Journal of mathematical control and information},
  19:\penalty0 5--23, 2002.

\bibitem[Bellen and Guglielmi(1999)]{Bellen:99}
A.~Bellen and N.~Guglielmi.
\newblock Methods for linear systems of circuit delay differential equations of
  neutral type.
\newblock \emph{IEEE Transactions on Circuits and Systems I}, 76(1):\penalty0
  212--215, 1999.

\bibitem[Verriest and Pepe(2007)]{Verriest:07}
E.~I. Verriest and P.~Pepe.
\newblock Time optimal and optimal impulsive control for coupled differential
  difference point delay systems with an application in forestry.
\newblock In \emph{{IFAC} Workshop on time-delay systems}, Nantes, France,
  2007.

\bibitem[Ebihara et~al.(2012)Ebihara, Peaucelle, and Arzelier]{Ebihara:12}
Y.~Ebihara, D.~Peaucelle, and D.~Arzelier.
\newblock Optimal ${L}_1$-controller synthesis for positive systems and its
  robustness properties.
\newblock In \emph{American Control Conference}, pages 5992--5997, Montreal,
  Canada, 2012.

\bibitem[Briat(2017{\natexlab{b}})]{Briat:17NAHS}
C.~Briat.
\newblock Dwell-time stability and stabilization conditions for linear positive
  impulsive and switched systems.
\newblock \emph{Nonlinear Analysis: Hybrid Systems}, 24:\penalty0 198--226,
  2017{\natexlab{b}}.

\bibitem[Briat and Khammash(2017)]{Briat:17ifacObs}
C.~Briat and M.~Khammash.
\newblock Simple interval observers for linear impulsive systems with applications to sampled-data and switched systems.
\newblock In \emph{20th IFAC World Congress (Accepted)}, Toulouse, France,
  2017.

\end{thebibliography}

\end{document}